\documentclass[11pt]{article}

\usepackage[margin=1in]{geometry}  % set the margins to 1in on all sides
\usepackage{graphicx}              % to include figures
\usepackage{amsmath}               % great math stuff
\usepackage{amsfonts}              % for blackboard bold, etc
\usepackage{amsthm}                % better theorem environments
\usepackage{amssymb}
\usepackage{hyperref}
\usepackage{theoremref}
\usepackage{wrapfig}

\usepackage{graphicx}

\numberwithin{equation}{section}
\newtheorem{thm}{Theorem}[section]

\theoremstyle{definition}
\newtheorem{lemma}[thm]{Lemma}

\newtheorem{prop}[thm]{Proposition}

\theoremstyle{definition}
\newtheorem{corollary}[thm]{Corollary}

\theoremstyle{definition}
\newtheorem{remark}{Remark}

\newtheorem{mydef}{Definition}

\newcommand*\diff{\mathop{}\!\mathrm{d}}
\newcommand{\RR}{\mathbb{R}}      % for Real numbers
\newcommand{\ZZ}{\mathbb{Z}}      % for Integers
\newcommand{\CC}{\mathbb{C}}

\begin{document}

\title{\textbf{Hopf Bifurcation in Structural Population Models}}

\author{Narek Hovsepyan\thanks{University of Bonn, Germany}, \qquad Juan J. L. Vel\'azquez\thanks{IAM, Bonn, Germany}}

\maketitle

\begin{abstract}
We study a nonlinear PDE problem motivated by the peculiar patterns arising in myxobacteria, namely counter-migrating cell density waves. We rigorously prove the existence of Hopf bifurcations for some specific values of the parameters of the system. This shows the existence of periodic solutions for the systems under consideration.
\end{abstract}

\tableofcontents

\section{Introduction}

\qquad Pattern formation is ubiquitous in biological and chemical systems. Being able to distinguish between the possible underlying mechanisms driving these patterns and their related functions, is an important aim for a better understanding and for experimental control. Considering patterns generated by diffusive instabilities , in his pioneering work (cf. \cite{turing}) Turing proved, that for chemical reactions diffusion can drive an otherwise stable system towards pattern formation with a characteristic wavelength or characteristic time period.

\quad There are, however, structure forming processes in biology, where diffusive signals do not seem to play the major role. An example for this are counter migrating rippling waves in populations of myxobacteria which occur before their final aggregation and fruiting body formation, \cite{dworkin}. These waves are assumed to result from a local (non-diffusive), i.e. cell-cell contact induced, exchange of a so-called C-signal. Further, the aggregation process of myxobacteria happens during a state without a cell division, so mass is conserved.

\subsection{Pattern forming equations}

We consider a linearized equation 
\begin{equation} \label{pattern formation equation}
\partial_t y = By
\end{equation}
where $B$ is a linear operator, invariant under translations, i.e. $B[y(\cdot + a)]=B[y](\cdot + a)$, $a \in \RR^N$. The function $x \mapsto y(x,t)$ maps $\RR^N$ into a suitable function space $Z$, describing the variables needed to characterize a "macroscopic" region $[x, x+\diff x]$. Typically $Z$ will include chemical concentrations, internal cell variables, cell orientations and others. It is well known that this type of operators can be analyzed by Fourier analysis, i.e. for $k\in \RR^N$ consider $B(e^{ikx}V)=\left[ \tilde{B}(k)V \right] e^{ikx}, \quad V\in Z$. Where $\tilde{B}(k)$ is a linear operator acting on $Z$. We can then look for solutions of \eqref{pattern formation equation} of the form $y=e^{z t +ikx}V$, where $V$ is an eigenfunction of $z V=\tilde{B}(k) V$. Under some general compactness assumptions, the eigenvalues of the latter are a discrete set $\{z_1(k),z_2(k),...\}$ for each $k\in \RR^N$. For a perturbation $y(x,0)=y_0(x)=Ve^{ikx}$ with wavenumber $k$ one can calculate its corresponding growth rate $\Omega(k):= \max_j Re(z_j(k))$.

\begin{mydef} \label{pattern formation}
Equation \eqref{pattern formation equation} is said to generate patterns, if $\Omega(k)$ achieves a global maximum at a finite number of nonzero values $k_i$, $i=1,...,l$. In this case solutions of \eqref{pattern formation equation} with suitable initial data develop patterns with wavelength $\lambda_i=\frac{2\pi}{k_i}$.
\begin{itemize}
\item if $Im(z_j(k_i))\neq 0$ for some $i$ and $j$, then we say that \eqref{pattern formation equation} generates oscillatory patterns.
\item if $Im(z_j(k_i))= 0$ for every $i$ and $j$, then we say that \eqref{pattern formation equation} generates stationary patterns.
\end{itemize}
\end{mydef}

\subsection{Turing's instabilities}

Let us briefly recall the instability results derived by Turing in \cite{turing} for reaction-diffusion systems within the above mentioned framework. Here we restrict ourselves to one dimension. For this case in \cite{turing} the pattern-forming properties of equations of type 
\begin{equation} \label{Tur}
\partial_t y = D\partial_x^2 y + Ay
\end{equation}
were studied, where $y=y(x,t)$ has values in $\mathbb{R}^N$ and $D,A \in M_N(\mathbb{R})$ are real $N\times N$ matrices, $D$ is a diagonal matrix with diagonal entries $D_i > 0$ and $A = (a_{ij})$. System \eqref{Tur} can be obtained from the linearization of a reaction-diffusion system without cross-diffusion terms near a homogeneous state.

\begin{thm} (Turing, cf. \cite{turing})

\begin{itemize}

\item For $N = 1$ equation \eqref{Tur} does not generate patterns for any $A \in \mathbb{R}$

\item For $N = 2$ system \eqref{Tur} generates stationary patterns if 
\begin{equation} \label{StPatterns}
a_{11}+a_{22}<0, \quad \det A>0, \quad a_{11}D_2+a_{22}D_1 >2\sqrt{D_1D_2 \det A} > 0
\end{equation}
On the other hand, \eqref{Tur} doesn't generate oscillatory patterns for any $A \in M_2(\mathbb{R})$

\item For $N = 3$ there exists an open set of matrices $A \in M_3(\mathbb{R})$ such that \eqref{Tur} generates oscillatory patterns.

\end{itemize}

\end{thm}

This means that linear reaction-diffusion equations can generate nontrivial patterns with specific wavelengths, if at least two species  are involved. Moreover patterns with nontrivial characteristic length and time scales can be generated, if at least three species are involved. It is well known that conditions \eqref{StPatterns} can be interpreted as the interplay between a short range acting chemical activator and a long range acting chemical inhibitor, with the diffusion coefficient of the inhibitor being larger than the one of the activator, cf. \cite{gierer}.

\subsection{Model without diffusive interactions}
We consider a problem motivated by the intriguing counter migrating wave-like patterns observed before the final aggregation of and self-organization of myxobacteria (cf. \cite{dworkin}) which happens under starvation conditions. During their alignment and before their final self-organization takes place, the bacteria move in opposite directions in a quasi one-dimensional fashion and reverse their direction of motion, mainly due to contact and exchange of a so-called C-signal with counter migrating cells. As a result, counter-migrating population waves with a characteristic wavelength occur.

From \cite{lutscher} it is known that one cell state for each direction of motion is not sufficient to decide about pattern formation on the linearized level. Therefore we introduce 4 states. Consider bacteria, which exist in two different states $1$ and $2$. Let $u_i, v_i$ denote the densities of cells which move towards the right, respectively the left, with internal state $i = 1, 2$. First, the bacteria change their state from $1$ to $2$, e.g. from a non-excited state to an excited state. Then, in a second step, they reverse their direction of motion. So we assume that there exists an intermediate state for the cells before they reorient. This can be interpreted e.g. by the local transfer of the so-called C-signal during cell-cell contact, which excites the bacterium and/or prepares it to switch the location of its molecular motor for movement, before it reverses its direction. So the four cellular states evolve according to the following transition: $u_1 \rightarrow u_2 \rightarrow v_1 \rightarrow v_2 \rightarrow u_1$.

Translating the above-described kinetics into a system of differential equations we obtain
\begin{equation} \label{1}
\begin{aligned}
\partial_t u_1 + \partial_x u_1 = S_2(u_1, u_2, v_1, v_2) - T_1(u_1, u_2, v_1, v_2) \\
\partial_t u_2 + \partial_x u_2 = T_1(u_1, u_2, v_1, v_2) - T_2(u_1, u_2, v_1, v_2) \\
\partial_t v_1 - \partial_x v_1 = T_2(u_1, u_2, v_1, v_2) - S_1(u_1, u_2, v_1, v_2) \\
\partial_t v_2 - \partial_x v_2 = S_1(u_1, u_2, v_1, v_2) - S_2(u_1, u_2, v_1, v_2)
\end{aligned}
\end{equation}
To further simplify, we assume that the system is invariant under the change of variables $(x, u_i, v_i) \rightarrow (-x, v_i, u_i) \quad i = 1, 2$. Therefore $T_i(u_1, u_2, v_1, v_2) = S_i(v_1, v_2, u_1, u_2)$. So \eqref{1} can be rewritten just in terms of $T_1, T_2$ accordingly. Linearizing a system of the above type around a homogeneous equilibrium one obtains a linearization of the form (cf. \cite{velaz}):
\begin{equation} \label{Linear}
\partial_t y + U \cdot \partial_x y + D A y = 0, \quad \text{where}
\end{equation}
\begin{equation} \label{D,U}
D = 
\begin{pmatrix}
1 & 0 & \ldots & -1 \\
-1 & 1 & \ldots & 0 \\
& \cdots & \cdots  \\
0 & \ldots & -1 & 1
\end{pmatrix}, \qquad U = 
\begin{pmatrix}
U_1 & 0 & \ldots & 0 \\
0 & U_2 & \ldots & 0 \\
& \cdots & \cdots \\
0 & \ldots & 0 & U_N
\end{pmatrix}
\end{equation}
\\
Here $D$ describes the transition between the states, $U$ represents the velocities of bacteria in a particular state and moving in a particular direction and $A$ is a square matrix in $M_N(\mathbb{R})$. The space of internal cell states is the set $\{1, 2, \ldots ,N\}$. The following properties of the matrix $DA$ are relevant.

\begin{prop} \label{properties of D}
Let $D$ be given as in \eqref{D,U}. Then, the matrix $DA$ has a zero eigenvalue. Moreover
\begin{equation} \label{b}
b := (1,\ldots,1)^t
\end{equation}
is an element of the kernel of $(DA)^t$, which is the transposed matrix of $DA$.
\end{prop}

\begin{proof}
We have $\det (DA) = \det D \cdot \det A = 0$ since $\det D = 0$. Hence $0$ is an element of the spectrum of $DA$. Since $D^t b = 0$ \[(DA)^t b = A^t D^t b = A^t \cdot 0 = 0\]
Thus $b \in \ker (DA)^t$.
\end{proof}

To obtain a class of matrices $A$ for which \eqref{Linear} exhibits nontrivial patterns when $U$ is nondegenerate (i.e. $U_i \neq U_j$ for any $i \neq j$) authors of \cite{velaz} choose $A = A_0 + \delta M$, where $A_0$ yields a "hyperbolic" dispersion relation for \eqref{Linear}, i.e. the most unstable part of the spectrum of $A_0$ lies on the imaginary axis. Note that for a pure transport equation (first order hyperbolic equation) the spectrum is the imaginary axis. The matrix $\delta M$ will then be chosen as a small perturbation of $A_0$ that will deform that part of the spectrum into a curve which yields pattern formation. (Patterns are generated by the "hyperbolic" part and not by the diffusive part as in the Turing's model). So pattern forming solutions bifurcate from the non-pattern forming state $\delta = 0$.

For a typical example generating oscillatory patters we obtain \eqref{Linear} by linearizing \eqref{1} with $A = A_0 + \delta M$ where
\begin{equation} \label{A0,M}
A_0 = 
\begin{pmatrix}
0 & -1 & 0 & 0 \\
0 & -1/2 & 0 & 0 \\
0 & 0 & 0 & -1 \\
0 & 0 & 0 & -1/2
\end{pmatrix}, \qquad
M =
\begin{pmatrix}
1 & 0 & 1 & 1 \\
0 & 0 & 0 & 0 \\
1 & 1 & 1 & 0 \\
0 & 0 & 0 & 0
\end{pmatrix}
\end{equation}
Further $N = 4$ and 
\begin{equation} \label{Ui}
U_1 = 2, \quad U_2 = 1, \quad U_3 = -2, \quad U_4 = -1 \end{equation}
Then for $\delta > 0$ sufficiently small the differential equation \eqref{Linear} generates oscillatory patterns (cf. \cite{velaz}, Theorem 5.4).

\begin{remark}
By \eqref{Ui} we assumed that the cells in the excited state move with slower speed than the non-excited cells. In \cite{velaz} it was established that if the bacteria in the excited state move with the same speed as the non-excited ones, then no solutions with oscillatory patterns bifurcate from "hyperbolic" matrices.
\end{remark}

Our model is motivated by both the Turing's model and the model without diffusive interactions generating oscillatory patterns, described above. We have started by considering a modification of the model problem \eqref{Linear} with $D, A$ defined by \eqref{D,U}, \eqref{A0,M}, \eqref{Ui} by adding a "small" diffusive interaction and a nonlinearity which preserves the total mass. However, it turns out that choosing $A$ as  in \eqref{A0,M} violates Hopf's nonresonance condition (cf. Remark~\ref{choice of M_s}). So we have modified $A$ accordingly. In fact we consider two different models corresponding to two choices of the matrix $A$: 
\begin{enumerate}
\item[1)] symmetric, in which case the system is reflection invariant (cf. Corollary~\ref{symmetries}). Now the latter causes some degeneracies by breaking the simplicity assumption of the purely imaginary eigenvalue (cf. \eqref{spectral of F} and Proposition~\ref{degeneracy caused by the symmetry}) and as a result the theory of bifurcation at multiple eigenvalues should be applied (cf. \cite{kiel2}). This case is treated in the Section~\ref{symmetric model} with preliminaries given in the Section~\ref{Hopf bifurcation at multiple eigenvalues, section}.

\item[2)] nonsymmetric, in which case the system isn't reflection invariant anymore (not even at a linear level) and the classical theory of the Hopf bifurcation applies. This case is treated in the Section~\ref{nonsymmetric model} with preliminaries given in the Section~\ref{Hopf bifurcation at single eigenvalues}.  
\end{enumerate}

\section{Description of the model and main results}
As we have already mentioned in the introduction, we start by considering the problem
\begin{equation} \label{eq before change}
\partial_t y = -U \partial_x y - DA y + \varepsilon \partial_x^2 y + Q(y) \quad \text{on} \ [0, \lambda]
\end{equation}
where $\lambda > 0$ (is the bifurcation parameter), \ $y = (y_1, y_2, y_3, y_4)$ and $\varepsilon > 0$ is a small parameter. The matrices $U, \ D$ are chosen as in \eqref{D,U}, \eqref{Ui} and $A = A_0 + \delta M$ where $A_0$ is chosen as in \eqref{A0,M} and $M$ can be one of the two alternatives given below depending on the model (nonsymmetric or symmetric), i.e.
\begin{equation} \label{D,U, A_0 4x4}
D =
\begin{pmatrix}
1 & 0 & 0 & -1 \\
-1 & 1 & 0 & 0 \\
0 & -1 & 1 & 0 \\
0& 0& -1 & 1
\end{pmatrix}, \quad U = 
\begin{pmatrix}
2 & 0 & 0 & 0 \\
0 & 1 & 0 & 0 \\
0 & 0 & -2 & 0 \\
0 & 0 & 0 & -1   
\end{pmatrix}, \quad A_0 =
\begin{pmatrix}
0 & -1 & 0 & 0 \\
0 & -1/2 & 0 & 0 \\
0 & 0 & 0 & -1 \\
0 & 0 & 0 & -1/2
\end{pmatrix}
\end{equation}

\begin{equation} \label{M_ns and M_s}
\quad M_{ns} = 
\begin{pmatrix}
0 & 1 & 1 & 0 \\
0 & 1/2 & 0 & 0 \\
1 & 0 & 0 & 2 \\
0 & 0 & 0 & -1/2   
\end{pmatrix}, \quad M_s=
\begin{pmatrix}
1 & 0 & 1.1 & 1 \\
0 & 0 & 0 & 0 \\
1.1 & 1 & 1 & 0 \\
0 & 0 & 0 & 0   
\end{pmatrix}
\end{equation}
$Q(\cdot)$ is a non-linear term which would be specified later. 
\begin{remark}
The parabolic diffusion term $\varepsilon \partial_x^2 y$ has been added for technical reasons in order to have better regularity. The classical theory of bifurcation (both for single and for multiple eigenvalues) breaks down when $\varepsilon=0$. Mathematical models for pattern formation in myxobacteria (with difussion) can be found in \cite{neu} or \cite{deutsch}.     
\end{remark}

\begin{remark}
General conditions for existence of purely imaginary eigenvalues and for oscillatory behavior of the corresponding system were obtained in \cite{velaz}. We used these to specify the matrices under consideration. Otherwise, the latter would had been very hard to obtain.
\end{remark}

We impose periodic boundary conditions on solutions of \eqref{eq before change}
\begin{equation} \label{BC before change}
\begin{aligned}
y(t, x + \lambda) \equiv y(t, x) \\
\partial_x y(t, x + \lambda) \equiv \partial_x y(t, x)
\end{aligned}
\end{equation}
\textit{\underline{Eigenvalues of the problem}} First, we consider the linearized version of our problem by just dropping the nonlinearity $Q$:
\begin{equation} \label{linear eq before change}
\partial_t y = -U \partial_x y - DA y + \varepsilon \partial_x^2 y \quad \text{on} \ [0, \lambda]
\end{equation}
To find the eigenvalues of \eqref{linear eq before change} we make a separation of variable ansatz: $y(t, x) = e^{zt} e^{ikx} v$, where $z \in \mathbb{C}$ (is called an eigenvalue), $k \in \mathbb{R}$, $i$ is the imaginary unit and $v \in \mathbb{C}^4$ is a constant vector.
Next we plug the ansatz into \eqref{linear eq before change} to obtain:
\begin{equation} \label{eigenvalue problem 1v}
\underbrace{(-ikU - DA - \varepsilon k^2 Id)}_{ \text{\normalsize{\textit{M(k)}}}}v = z v
\end{equation}
Thus we obtained an eigenvalue problem depending on the parameter $k$. Here $Id$ is the $4 \times 4$ identity matrix which would not be mentioned explicitly in the sequel. Let us denote the solutions of this problem by
\begin{equation} \label{solutions of eigenvalue problem 1v}
z_1 = z_1(k), \quad z_2 = z_2(k), \quad z_3 = z_3(k), \quad z_4 = z_4(k)
\end{equation} 
Plotting these functions numerically we see from figures below that one of the eigenvalues crosses the imaginary axis (with non-vanishing speed): 
\begin{equation} \label{bifurcation point}
\exists k_0 \in \mathbb{R} \quad \text{s. t.} \quad z_1(k_0) = i \kappa_0 \quad \text{for} \quad \kappa_0 < 0 
\end{equation}
\begin{figure}[h]
    \includegraphics[width=7.5cm, height=5cm]{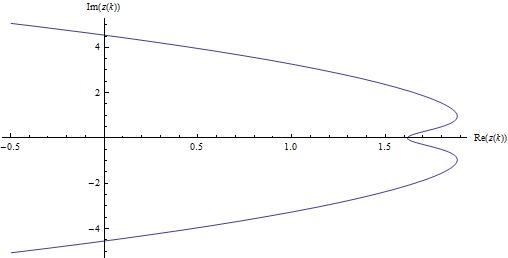} \quad
\includegraphics[width=7.5cm, height=5cm]{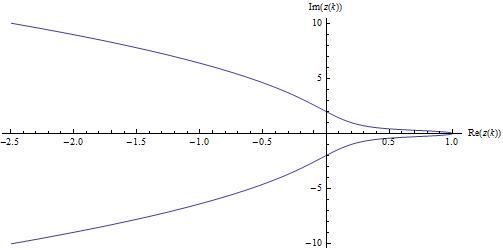}
    \caption{Eigenvalues $z_1$ (left) and $z_2$ (right) for the non-symmetric model ($\varepsilon=0.1, \delta=1$). The other eigenvalues have negative real parts. Horizontal: $Re(z(k))$, vertical: $Im(z(k))$}
    \label{fig:non-symmetric model}
\end{figure}

\begin{figure}[h]
    \includegraphics[width=7.5cm, height=5cm]{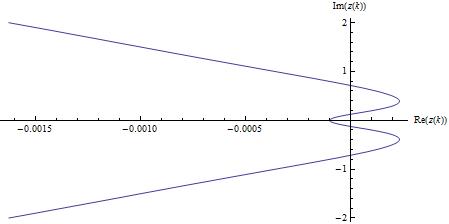} \quad
\includegraphics[width=7.5cm, height=5cm]{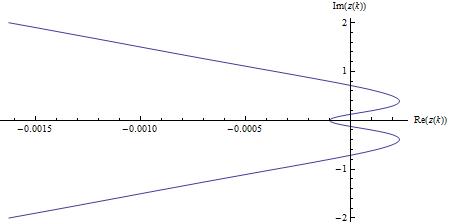}
    \caption{Eigenvalues $z_1$ (left) and $z_2$ (right) for the symmetric model ($\varepsilon=\delta=0.001$). The other eigenvalues have negative real parts. Horizontal: $Re(z(k))$, vertical: $Im(z(k))$}
    \label{fig:symmetric model}
\end{figure}

In particular there is a certain value of $k$ at which the eigenvalue is purely imaginary (we take $\kappa_0$ to be the one with the largest modulus). However $k$ is not arbitrary since $y$ should also satisfy \eqref{BC before change}. This yields
\begin{equation} \label{formula for k}
k = \dfrac{2 \pi n}{\lambda}, \quad n \in \mathbb{Z}
\end{equation}
So we see that there is a certain value of $\lambda$ for which one of the eigenvalues is purely imaginary (this is why $\lambda$ is called a bifurcation parameter). 

We would like to consider the problem on a fixed interval and take the bifurcation parameter into the equation. To that end we make the following change of variables:
\[ \xi := \dfrac{x}{\lambda} \quad \text{and} \quad y(t, x) = f(t, \xi), \quad \text{we then see} \quad x \in [0, \lambda] \quad \Longleftrightarrow \quad \xi \in [0, 1]
\]
So the problem translates into (we rename $\xi$ by $x$ and $f$ by $y$):
\begin{equation} \label{main problem}
\begin{aligned}
\begin{cases}
\partial_t y = -\dfrac{1}{\lambda}U \partial_x y - DA y + \dfrac{\varepsilon}{\lambda^2} \partial_x^2 y + Q(y) \quad \\
y(t, x+1) = y(t, x) \\
\partial_x y(t, x+1) = \partial_x y(t, x)
\end{cases}
\end{aligned}
\ \ \forall t > 0, \ x \in [0, 1]
\end{equation}
\textit{\underline{Eigenvalues of the problem}} Making the same ansatz as above, for \eqref{main problem} we obtain the following eigenvalue problem:
\begin{equation} \label{eigenvalue problem 2v}
\underbrace{\left( -\varepsilon\dfrac{4 \pi^2 n^2}{\lambda^2} - \dfrac{2 \pi n}{\lambda} iU - DA \right)}_{\text{\normalsize{$\widetilde{M}(n, \lambda)$}}} v = z v
\end{equation}
whose solutions we denote by:
$\tilde{z_1}=\tilde{z_1}(n, \lambda), \quad \tilde{z_2}=\tilde{z_2}(n, \lambda), \quad \tilde{z_3}=\tilde{z_3}(n, \lambda), \quad \tilde{z_4}=\tilde{z_4}(n, \lambda)$.
We note that
\begin{equation} \label{connecting 2 eigenvalues}
\widetilde{M}(n, \lambda) = M(\dfrac{2 \pi n}{\lambda}) \quad \text{and} \quad
\tilde{z_j}(n, \lambda) = z_j(\dfrac{2 \pi n}{\lambda}) \qquad \text{for} \quad j = 1,...,4
\end{equation}
Define a point $\lambda_0$ and using \eqref{bifurcation point} note that
\begin{equation} \label{lambda0}
\lambda_0 := \dfrac{2 \pi}{k_0}, \qquad \tilde{z_1}(1, \lambda_0)=i \kappa_0
\end{equation}
We now formulate the main theorems:
\begin{thm}(Bifurcation, nonsymmetric model) \label{main theorem, nonsymmetric}

\hspace{-1.9em} Consider the parameter-dependent evolution equation
\begin{equation} \label{main problem evolution}
\partial_t y = F(y, \lambda), \quad with
\end{equation}
\begin{equation} \label{our F}
F(y, \lambda) = -\dfrac{1}{\lambda}U \partial_x y - DA y + \dfrac{\varepsilon}{\lambda^2} \partial_x^2 y + Q(y) \ : X \times \mathbb{R} \rightarrow Z, \quad \text{where}
\end{equation}
\begin{equation} \label{our X, Z}
\begin{aligned}
&X := \{ \varphi \in [H^2(0, 1)]^4 \ / \ \varphi(0) = \varphi(1), \ \varphi'(0) = \varphi'(1) \ \text{and} \ \sum_{j=1}^4 \int_0^1 \varphi_j \diff x = 0 \} \\
&Z := \{ \varphi \in [L^2(0,1)]^4 \ / \ \sum_{j=1}^4 \int_0^1 \varphi_j \diff x = 0\}
\end{aligned}
\end{equation}
with the notation $\varphi = (\varphi_1,...,\varphi_4)$. Moreover let 
\begin{equation} \label{our Q}
Q(y) = \begin{pmatrix}
c_4 y_4^2 - c_1 y_1^2 \\
c_1 y_1^2 - c_2 y_2^2 \\
c_2 y_2^2 - c_3 y_3^2 \\
c_3 y_3^2 - c_4 y_4^2
\end{pmatrix} \quad \text{with} \ c_j \in \mathbb{R} \ \text{for} \ j=1,...,4
\end{equation}
Finally let $D, U$ and $A=A_0+\delta M_{ns}$ be defined according to \eqref{D,U, A_0 4x4} with $\delta=1$, let $\varepsilon > 0$ be a small parameter and let $\lambda_0$ and $\kappa_0$ be defined by \eqref{bifurcation point}, \eqref{lambda0}.  

Then there exists a continuously differentiable curve $\{(y(r), \lambda (r))\}$ of (real) $\dfrac{2 \pi}{\kappa (r)}$ - periodic solutions of \eqref{main problem evolution} passing through $(y(0), \lambda (0)) = (0, \lambda_0)$ with $\kappa(0) = \kappa_0$ in \\ ${\text{\large{( C}}}^{ \ 1+\alpha}_{{2 \pi} / \kappa (r)}\text{\large{($\mathbb{R}$, Z)}} \ \bigcap \ {\text{\large{C}}}^{ \ \alpha}_{{2 \pi} / \kappa (r)}\text{\large{($\mathbb{R}$, X) )}} \times \text{\large{$\mathbb{R}$}}$. Every other periodic solution of \eqref{main problem evolution} in a neighborhood of $(0, \lambda_0)$ is obtained from $(y(r), \lambda (r))$ by a phase shift $S_\theta y(r)$. 
\end{thm}

\begin{thm}(Bifurcation, symmetric model) \label{main theorem, symmetric}

\hspace{-1.9em} Consider the parameter-dependent evolution equation
\begin{equation} \label{main problem evolution symmetric}
\partial_t y = F(y, \lambda)
\end{equation}
where $F, X, Z$ are given by \eqref{our F} and \eqref{our X, Z}. The nonlinearity $Q$ given by \eqref{our Q}, satisfies $c_1 = c_3 = 1$ and $c_2 = c_4 = 0$.
The matrices $D, U$ and $A=A_0 + \delta M_s$ are given by \eqref{D,U, A_0 4x4} with $\delta>0$ small. Moreover let $\varepsilon=\varepsilon(\lambda)= \frac{\lambda^2}{\lambda_0^2}\varepsilon_0$ with $\varepsilon_0>0$ being a small parameter and let $\lambda_0$ and $\kappa_0$ be defined by \eqref{bifurcation point}, \eqref{lambda0}.  

Then there exists a nontrivial solution curve $\{(y(r), \lambda (r))\}$ of periodic solutions of \eqref{main problem evolution} passing through $(0, \lambda_0)$ and emanating in the direction $0 \neq v^0 \in \RR^4$.  
 
\end{thm}

\hspace{-1.9em} We make a few remarks:
\begin{remark} (Pattern formation) \\
The linearized problem $\partial_ty=D_yF(0,\lambda)y$ generates oscillatory patterns for both of the non-symmetric and symmetric models. In fact $\Omega (k)=\max_{j=1,..,4} Re z_j(k)$ has the shape given in the Figure~\ref{omega} and one can easily check that the first part of Definition~\ref{pattern formation} is satisfied.
\begin{figure}[h]
    \includegraphics[width=7.5cm, height=5cm]{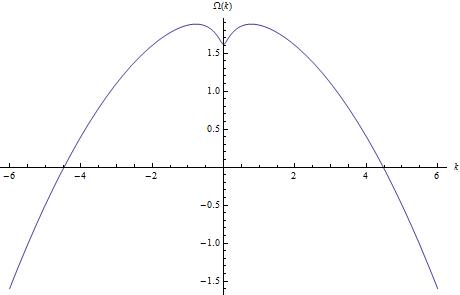} \quad
\includegraphics[width=7.5cm, height=5cm]{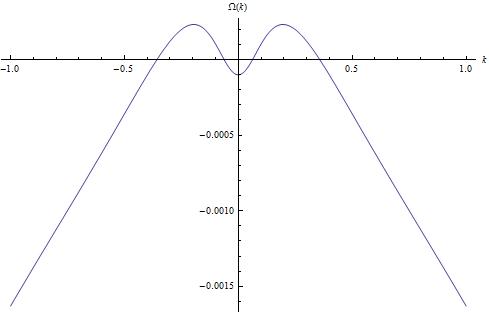}
    \caption{$\Omega(k)$ for the nonsymmetric (left) and symmetric (right) models}
    \label{omega}
\end{figure}
\end{remark}

\begin{remark}
Of course it is possible to consider other choices of constants $c_1, c_2$ of the nonlinearity $Q$ in Theorem~\ref{main theorem, symmetric}. However one should note the difficulty of solving a $4$-dimensional system of cubic equations with one free parameter (cf. \eqref{necessary equation multiple eigenvalue}).
\end{remark}

\begin{remark}(Choice of the spaces $X$ and $Z$) \\
One might think that a natural choice for the spaces $X$ and $Z$ is $[H^2_{per}(0,1)]^4$ and $[L^2(0,1)]^4$ rather than \eqref{our X, Z}. However for this choice the nonresonance condition \eqref{nonresonance} is violated: by Proposition~\ref{properties of D} the matrix $DA$ always has a $0$ eigenvalue for any choice of $A$, on the other hand the eigenvalues of the matrices $\widetilde{M}(n, \lambda_0), \ n \in \ZZ$ are also eigenvalues of $D_y F(0, \lambda_0)$ (cf. Proposition~\ref{degeneracy caused by the symmetry}), but $\widetilde{M}(0, \lambda_0) = - DA$. So the nonresonance condition is violated for $n = 0$, no matter how we choose $A$. This suggests to consider special subspaces of $[H^2_{per}]^4$ and $[L^2]^4$ (based on symmetry/conservation properties of \eqref{main problem evolution}) on which $DA$ would be invertible.
\end{remark}

\begin{remark} (Choice of the nonlinearity) \\
Since the linearized problem $\partial_t y = -\dfrac{1}{\lambda}U \partial_x y - DA y + \dfrac{\varepsilon}{\lambda^2} \partial_x^2 y$ has a mass conservation property (cf. Corollary~\ref{symmetries}) we chose the nonlinearity $Q$ so that this property remains valid. In fact for the nonsymmetric model the result, i.e. Theorem~\ref{main theorem, nonsymmetric} is valid for any nonlinearity $Q$ not violating mass conservation and satisfying $DQ(0) = 0$. Our choice is relevant since it also has a reflection invariance property (when choosing $c_j$ appropriately).   
\end{remark}

\begin{lemma} \label{symmetries of the operator}
Consider the problem \eqref{main problem} and let $F$ be given by \eqref{our F} then

\begin{enumerate}
\item[\textit{(i)}] if $y = (y_1, y_2, y_3, y_4)$ satisfies the boundary conditions $y(\cdot, 0)=y(\cdot, 1)$ and $\partial_x y(\cdot, 0) = \partial_x y(\cdot, 1)$ then 
\begin{equation} \label{components of F add to 0}
\sum_{j=1}^4 \int_0^1 [F(y, \lambda)]_j \diff x = 0
\end{equation}

\item[\textit{(ii)}] define 
\begin{equation} \label{P and W :symmetries}
P =
\begin{pmatrix}
0 & 0 & 1 & 0 \\
0 & 0 & 0 & 1 \\
1 & 0 & 0 & 0 \\
0 & 1 & 0 & 0 
\end{pmatrix}
\quad \text{and} \quad W y(\cdot, x) := P y(\cdot, 1-x) \quad \text{then}
\end{equation}
\begin{equation} \label{symmetry of F}
W F W = F 
\end{equation}
provided $c_1 = c_3, \ c_4 = c_2$ (cf. \eqref{our Q}) and the matrix $A = (a_{ij})$ satisfies
\begin{equation} \label{reflection properties for A}
\begin{pmatrix}
a_{11} \\ a_{12} \\ a_{13} \\ a_{14}
\end{pmatrix} -
\begin{pmatrix}
a_{41} \\ a_{42} \\ a_{43} \\ a_{44}
\end{pmatrix} =
\begin{pmatrix}
a_{33} \\ a_{34} \\ a_{31} \\ a_{32}
\end{pmatrix} -
\begin{pmatrix}
a_{23} \\ a_{24} \\ a_{21} \\ a_{22}
\end{pmatrix} \ \text{and} \
\begin{cases}
a_{21} + a_{23} = a_{41} + a_{43} \\
a_{22} + a_{24} = a_{42} + a_{44} 
\end{cases}
\end{equation}
\end{enumerate}
\end{lemma}

\begin{proof}
\begin{enumerate}
\item[\textit{(i)}] Let $(\cdot)_j$ or $[\cdot]_j$ denote the $j$-th component of a vector, then
\begin{equation*}
\sum_{j=1}^4 \int_0^1 [F(y, \lambda)]_j \diff x
= \sum_{j=1}^4 \int_0^1 \left[ - \dfrac{1}{\lambda} (U \partial_x y)_j - (DAy)_j + \dfrac{\varepsilon}{\lambda^2} \partial_x^2 y_j + Q(y)_j \right] \diff x
\end{equation*}
Using the boundary conditions we now show that the right-hand side of the above identity is $0$, which would conclude the proof.
\begin{align*}
\int_0^1 (U \partial_x y)_j \diff x & = \int_0^1 \sum_{k=1}^4 U_{jk} \ \partial_x y_k \diff x = \sum_{k=1}^4 U_{jk} \cdot [y_k(t, 1) - y_k (t, 0)] = 0 
\\
\int_0^1 \partial_x^2 y_j \diff x & = \partial_x y_j(t, 1) - \partial_x y_j(t, 0) = 0
\\
\sum_{j=1}^4 \int_0^1 Q(y)_j \diff x & = \int_0^1 \sum_{j=1}^4 Q(y)_j \diff x =\{\text{cf. \eqref{our Q}}\} = 0
\\
\sum_{j=1}^4 \int_0^1 (DAy)_j \diff x & = \int_0^1 \sum_{j=1}^4 (DAy)_j \diff x = \int_0^1 b \cdot DAy \diff x = \int_0^1 D^t b \cdot Ay \diff x = 0
\end{align*}
where the last equality holds since $D^t b = 0$ (cf. Proposition~\ref{properties of D}).

\item[\textit{(ii)}] Clearly $W^2 = Id$, so we want to show that $F$ commutes with $W$. We prove this property for each of the components of the definition of $F$. In the following we suppress the $t$-dependence from the notation and replace $\partial_x$ by $'$.
\begin{equation} \label{commuting parts with W}
\begin{aligned}
WU (Wy(x))' & =  - WU P y'(1-x) = - PUP y'(x) = \{\text{note} \ PUP=-U \} = U y'(x) \\
W (Wy)'' & = W P y''(1-x) = P^2 y''(x) = y''(x) \\
W Q(Wy) & = W 
\begin{pmatrix}
c_4 y_2^2 - c_1 y_3^2 \\
c_1 y_3^2 - c_2 y_4^2 \\
c_2 y_4^2 - c_3 y_1^2 \\
c_3 y_1^2 - c_4 y_2^2
\end{pmatrix}(1-x) = 
\begin{pmatrix}
c_2 y_4^2 - c_3 y_1^2 \\
c_3 y_1^2 - c_4 y_2^2 \\
c_4 y_2^2 - c_1 y_3^2 \\
c_1 y_3^2 - c_2 y_4^2
\end{pmatrix}(x) = Q(y) \\
WDAWy & = PDAPy(x)=DAy(x)
\end{aligned}
\end{equation}
Where it is easy to show that the last equality holds if and only if $A$ satisfies \eqref{reflection properties for A}. Thus we have obtained $WF(Wy)=F(y)$ 
\end{enumerate}
\end{proof}

\begin{remark} Note that $(i)$ shows that the operator $F$ (cf. \eqref{our F}) is well defined: it takes $[H^2_{per} (0,1)]^4$ to $Z$. For any $y \in [H^2_{per}]^4$ we get $||F(y, \lambda)||_{L^2} < \infty$ simply because in particular $y$ is bounded on $[0,1]$. 
\end{remark}

\vspace{8mm}

\begin{corollary} (Symmetries) \label{symmetries}
The problem \eqref{main problem} has the following symmetry properties: let $y=(y_1,...,y_4)$ be its solution then
\begin{enumerate}
\item[\textit{(i)}] for any $a,b\in \RR$ \ $y(\cdot+a, \cdot+b)$ is also a solution. Where $y(t,x+b)$ is defined by periodic extension. \hfill{(\textit{Translation Invariance})}
\item[\textit{(ii)}] $\displaystyle{\sum\limits_{j=1}^4 \int\limits_0^1 y_j(t, x) \diff x} = const$ \hfill{(\textit{Mass Conservation})}

\item[\textit{(iii)}] the system \eqref{main problem}(a) is invariant under the change of variables $(x, y_1, y_3) \rightarrow (1-x, y_3, y_1)$ and $(x, y_2, y_4) \rightarrow (1-x, y_4, y_2)$,\hfill{(\textit{Reflection Invariance})} \\
provided $c_1 = c_3, \ c_4 = c_2$ (cf. \eqref{our Q}) and the matrix $A = (a_{ij})$ satisfies \eqref{reflection properties for A}

\end{enumerate}
\end{corollary}

\begin{proof}
\begin{enumerate}
\item[\textit{(ii)}] Since $y$ solves \eqref{main problem}(a) applying Lemma~\ref{symmetries of the operator} $(i)$ we get
\begin{equation*}
\partial_t \left( \sum_{j=1}^4 \int_0^1 y_j(t,x) \diff x \right) = \sum_{j=1}^4 \int_0^1 \partial_t y_j(t,x) \diff x = \sum_{j=1}^4 \int_0^1 [F(y, \lambda)]_j \diff x = 0
\end{equation*}

\item[\textit{(iii)}] The change of variables $(x, y_1, y_3) \rightarrow (1-x, y_3, y_1)$ and $(x, y_2, y_4) \rightarrow (1-x, y_4, y_2)$ is exactly described by the operator $W$ (cf. \eqref{P and W :symmetries}) but then applying Lemma~\ref{symmetries of the operator} $(ii)$ we see
\begin{equation*}
\partial_t Wy(t,x) =P \partial_t y(t, 1-x) = P F(y(t,1-x), \lambda) = WF(y(t,x), \lambda) = F(Wy(t,x), \lambda) 
\end{equation*}
Thus if $y$ is a solution of \eqref{main problem} then so is its reflection $Wy$.  

\end{enumerate}
\end{proof}

\pagebreak

\begin{remark} \ \label{choice of M_ns}
\begin{enumerate}
\item The choice of $A$ with $M=M_{ns}$ (cf. \eqref{M_ns and M_s}) does not satisfy \eqref{reflection properties for A} hence the corresponding model doesn't have reflection invariance and is therefore called a nonsymmetric model. On the other hand the choice of $A$ with $M=M_s$ satisfies the reflection symmetry \eqref{reflection properties for A}, however violates the simplicity of the purely imaginary eigenvalue.
\item A typical example of a matrix $A$ satisfying \eqref{reflection properties for A} is 
\begin{equation*}
A =\begin{pmatrix}
R & T \\
T & R
\end{pmatrix}
\quad \text{where $R$ and $T$ are $2 \times 2$ real matrices.}
\end{equation*}
\end{enumerate}

\end{remark}

\begin{remark} (Choice of $M_s$)\label{choice of M_s} \\
The matrix $M_s$ in \eqref{M_ns and M_s} is a slight modification of the matrix $M$ in \eqref{A0,M}. The reason for this modification is that although the original choice has the required symmetry, it violates the nonresonance condition \eqref{nonresonance} by allowing $0$ to be an eigenvalue of $D_yF(0, \lambda_0)$ on $X$.  
\end{remark}

\vspace{10mm}
We now observe how high-dimensional kernels occur from a degeneration caused by the symmetry \eqref{symmetry of F}, and in particular the condition \eqref{spectral of F} is violated.

\begin{prop} (Degeneracy caused by the symmetry) \label{degeneracy caused by the symmetry}
\begin{enumerate}  

\item[\textit{(i)}] Let $F$ be given by \eqref{our F} and let $X, Z$ be defined according to \eqref{our X, Z}. Then the eigenvalues of the operator $D_y F(0, \lambda)$ on $X$ are $\tilde{z_1}(n, \lambda), \ \tilde{z_2}(n, \lambda), \ \tilde{z_3}(n, \lambda), \ \tilde{z_4}(n, \lambda)$ (cf. \eqref{connecting 2 eigenvalues}) for $n \neq 0$ and those $\tilde{z_j}(0, \lambda)$ which have a corresponding eigenvector

\begin{equation} \label{0 eigenvalue condition}
v \in \CC^4 \quad \text{with} \quad \sum_{j=1}^4 v_j = 0
\end{equation}

\item[\textit{(ii)}] If in addition the reflection symmetry holds, i.e. $c_1 = c_3, \ c_4 = c_2$ (cf. \eqref{our Q}) and the matrix $A = (a_{ij})$ satisfies \eqref{reflection properties for A} then the purely imaginary eigenvalue $i \kappa_0$ of $L_0 = D_y F(0, \lambda_0)$ has at least a two dimensional eigenspace
\begin{equation} \label{two dimensional eigenspace}
i \kappa_0 \longrightarrow \{e^{i2 \pi x}v_0, e^{-i2 \pi x}Pv_0\}
\end{equation}
where $P$ is defined in \eqref{P and W :symmetries}.
\end{enumerate} 
\end{prop}

\begin{proof}
\begin{enumerate}
\item[\textit{(i)}] Firstly, note that
\begin{equation} \label{first derivative DF(0,lambda)}
D_y F(0, \lambda) = - \dfrac{1}{\lambda} U \partial_x - DA + \dfrac{\varepsilon}{\lambda^2} \partial_x^2 \quad \in L(X, Z)
\end{equation}
Consider the Fourier mode $\varphi(x)=e^{i2 \pi nx}v$ with $n \in \ZZ$ and $v \in \CC^4$, then
\begin{equation}
\begin{split}
D_y F(0, \lambda) \varphi & = \left( -\varepsilon\dfrac{4 \pi^2 n^2}{\lambda^2} - \dfrac{2 \pi n}{\lambda} iU - DA \right)v e^{i2 \pi nx} = \\ 
& = \widetilde{M}(n, \lambda)v e^{i2 \pi nx} = \tilde{z_j}(n, \lambda) \varphi 
\end{split}
\end{equation}
provided $v$ is an eigenvector of $\widetilde{M}$ (cf. \eqref{eigenvalue problem 2v}). Note that, for $n \neq 0 \ \ \varphi \in X$, since
\begin{equation*}
\sum_{j=1}^4 \int_0^1 \varphi_j(t, x) \diff x = \sum_{j=1}^4 v_j \int_0^1 e^{i2 \pi nx} \diff x = 0 \quad \text{no matter what $v$ is}.
\end{equation*}
However, $\tilde{z_j}(0, \lambda)$ is an eigenvalue of $D_y F(0, \lambda)$ if it is an eigenvalue of $-DA$ with eigenvector $v \in \CC^4 $ satisfying $\sum\limits_{j=1}^4 v_j = 0$.
To prove that these are the only eigenvalues of our operator one simply uses Fourier expansion: for any $f \in X$ we may expand
\begin{equation} \label{Fourier expansion}
f(x) = \sum_{n \in \ZZ} \hat{f}(n)e^{i2 \pi nx} \quad \text{where} \quad \hat{f}(n)= \int_0^1 f(x)e^{-i2 \pi nx} \diff x \quad \in \CC^4  
\end{equation}

\begin{equation} \label{Fourier expansion of DF(0, lambda)}
D_y F(0, \lambda) f = \sum_{n \in \ZZ} D_y F(0, \lambda) \hat{f}(n)e^{i2 \pi nx} = \sum_{n \in \ZZ} \widetilde{M}(n, \lambda) \hat{f}(n) e^{i2 \pi nx}
\end{equation}
Now if we assume $\xi$ is an eigenvalue with eigenvector $f$ then we get
\begin{equation*}
\sum_{n \in \ZZ} \widetilde{M}(n, \lambda) \hat{f}(n) e^{i2 \pi nx} = \sum_{n \in \ZZ} \xi \hat{f}(n) e^{i2 \pi nx}
\end{equation*}
which implies $\widetilde{M}(n, \lambda) \hat{f}(n) = \xi \hat{f}(n)$ hence $\xi = \tilde{z_j}(n, \lambda)$ for some $j$ and $n$.

\item[\textit{(ii)}] From the definition of $\kappa_0$ (cf. \eqref{bifurcation point}, \eqref{lambda0}) and part $(i)$ it is now clear that $i \kappa_0$ is an eigenvalue of $D_y F(0, \lambda_0)$ with eigenvector $\varphi_0(x)=e^{i2 \pi x}v_0$:
\begin{equation} \label{eigenvalue i kappa_0 with eigenvector phi_0}
\begin{split}
& D_y F(0, \lambda_0) \varphi_0 = i \kappa_0 \varphi_0 \quad \text{with} \\
& \widetilde{M}(1, \lambda_0) v_0 = \left( -\varepsilon\dfrac{4 \pi^2}{\lambda_0^2} - \dfrac{2 \pi}{\lambda_0} iU - DA \right) v_0 = i \kappa_0 v_0
\end{split} 
\end{equation}
Let $W$ be the transformation operator \eqref{P and W :symmetries} then by identities \eqref{commuting parts with W} $D_y F(0, \lambda)$ also commutes with $W$, so
\begin{equation} \label{second eigenvector with W}
\begin{aligned}
& D_y F(0, \lambda_0) W \varphi_0 = W D_y F(0, \lambda_0) \varphi_0 = i \kappa_0 W \varphi_0 \quad \text{but} \\
& W \varphi_0 = e^{-i2 \pi x}Pv_0
\end{aligned}
\end{equation}
\end{enumerate}
\end{proof}

\begin{prop}(Nonresonance condition) \label{nonresonance condition Proposition} \\
Let $F$ be given by \eqref{our F} and let $X, Z$ be defined according to \eqref{our X, Z}. Then the nonresonance condition \eqref{nonresonance} is satisfied for both of the models.
\end{prop}

\begin{proof}
For any $n \in \mathbb{Z}\backslash \{-1, 0, 1\}$ by Proposition~\ref{degeneracy caused by the symmetry} $in\kappa_0$ cannot be an eigenvalue of $D_yF(0,\lambda_0)$, since the plots of the eigenvalues $z_j$ show that the imaginary axis is crossed at $\pm \kappa_0$ (with multiplicity $2$ for the symmetric model) and additionaly, crossed, say, at $\pm \eta_0$ with $|\eta_0|<|\kappa_0|$ (cf. Figures $1, 2$). Hence $\eta_0$ cannot be an integer multiple of $\kappa_0$. Remains to show that $0$ is not an eigenvalue on $X$. Suppose it is, then it should be an eigenvalue of the matrix $-DA$ with an eigenvector $0 \neq v \in \CC^4$ s.t. $\sum_{j=1}^4v_j =0$ (cf. \eqref{0 eigenvalue condition}). By Proposition~\ref{properties of D} $DAv=0 \Rightarrow Av=cb$ for some $c \in \CC$, which reads as follows (symmetric model: left, non-symmetric model: right)
\begin{equation*}
\begin{cases}
\delta v_1-v_2+1.1\delta v_3+\delta v_4=c \\
1.1\delta v_1+\delta v_2+\delta v_3-v_4=c \\
v_2=v_4=-2c
\end{cases} \qquad
\begin{cases}
(\delta-1)v_2+\delta v_3=c \\
\frac{1}{2}(\delta-1) v_2=c \\
\delta v_1 + (2\delta -1)v_4=c \\
-\frac{1}{2}(\delta+1) v_4=c
\end{cases}
\end{equation*}
The only solution (for both of the systems) satisfying $\sum_{j=1}^4v_j=0$ is $v=0$ which yields a contradiction.   
\end{proof}

\begin{prop}(Adjoint of $L_0$) \label{adjoint of A_0} \\
Let $F$ be given by \eqref{our F} and let $X, Z$ be defined according to \eqref{our X, Z}. Then the adjoint of $L_0=D_yF(0,\lambda_0):D(L_0)=X\rightarrow Z$ is given by
\begin{equation} \label{A_0^*}
L_0^*=\frac{1}{\lambda_0}U\partial_x - A^*D^* + \frac{1}{4}\int_0^1 \langle A^*D^*\cdot , b \rangle_2 \diff x \ b + \frac{\varepsilon}{\lambda_0^2}\partial_x^2 \quad : D(L_0^*)=X \rightarrow Z
\end{equation}
where $\langle \cdot, \cdot\rangle_2$ denotes Euclidean scalar product and $b=(1,1,1,1)$.
\end{prop}

\begin{proof}
The formula \eqref{A_0^*} can be checked using the integration by parts and \eqref{first derivative DF(0,lambda)}. The third term appeared in \eqref{A_0^*} in order to make sure that the image of $L_0^*$ lies in $Z$. The latter holds true since $\sum_{j=1}^4 \int_0^1 [A^*D^*\varphi]_j \diff x = \int_0^1 \langle A^*D^*\varphi, b \rangle_2 \diff x$ for any vector function $\varphi$. 
\end{proof}

\section{Preliminaries}  

\subsection{Hopf bifurcation at single eigenvalues} \label{Hopf bifurcation at single eigenvalues}

Here we follow \cite{kiel1}. Consider the parameter-dependent evolution equation 
\begin{equation} \label{evolution}
\frac{\diff y}{\diff t} = F(y, \lambda)
\end{equation}
in order to make sense of this evolution equation, we assume for the real Banach spaces $X$ and $Z$ that
\begin{equation} \label{embedding} 
X \subset Z \quad \text{is continuously embedded}
\end{equation}
and the derivative of $x$ with respect to $t$ is taken to be an element of $Z$. Further,
\begin{equation} \label{def of F}
\begin{aligned}
& F : U \times V \rightarrow Z, \quad \text{where} \\
& 0 \in U \subset X \quad \text{and} \quad \lambda_0 \in V \subset \mathbb{R} \quad \text{are open neighborhoods} \\
\end{aligned}
\end{equation}
The function $F$ is sufficiently smooth and in particular,
\begin{equation} \label{smoothness of F}
\begin{aligned}
& F(0, \lambda) = 0 \quad \text{and} \\
& D_y F(0,\lambda) \quad \text{exists in} \quad L(X,Z) \quad \text{for all} \quad \lambda \in V \\
& F \in C^3(U \times V, Z)
\end{aligned}
\end{equation}
We assume a trivial solution line $\{(0, \lambda)\ / \ \lambda \in \mathbb{R}\} \subset X \times \mathbb{R}$ for \eqref{evolution}, i.e. $F(0, \lambda) = 0$ for all $\lambda \in \mathbb{R}$. A bifurcation of nontrivial stationary solutions of \eqref{evolution} (i.e. of $F(y, \lambda) = 0$) can be caused by a loss of stability of trivial solution at $\lambda = \lambda_0$. More precisely, that loss of stability is described by a simple real eigenvalue of $D_y F(0,\lambda)$ leaving the "stable" left complex half-plane through $0$ at the critical value $\lambda = \lambda_0$ with "non-vanishing speed". Hopf bifurcation describes the effect of a loss of stability of the trivial solution of \eqref{evolution} via a pair of complex conjugate eigenvalues of $D_y F(0,\lambda)$ leaving the left complex half-plane through complex conjugate points on the imaginary axis at some critical value $\lambda = \lambda_0$. If $0$ is not an eigenvalue of $D_y F(0,\lambda_0)$, then by the Implicit Function Theorem, stationary solutions of \eqref{evolution} cannot bifurcate from the trivial solution line at $(0, \lambda_0)$. The Hopf Bifurcation Theorem, however, states that (time-) periodic solutions of \eqref{evolution} bifurcate at $(0, \lambda_0)$. Next assume,

\begin{equation} \label{spectral of F}
\begin{aligned}
& i \kappa_0 (\neq 0) \quad \text{is a simple eigenvalue of} \quad  D_y F(0,\lambda_0) \\ 
& \text{with eigenvector} \quad \varphi_0 \not\in \Re (i \kappa_0 I - D_y F(0,\lambda_0)), \\
& \pm i \kappa_0 I - D_y F(0,\lambda_0) \quad \text{are Fredholm operators of index zero}.
\end{aligned}
\end{equation}
Here $\Re(\cdot)$ denotes the range of a mapping. These guarantee existence of perturbed eigenvalues $z(\lambda)$ of $D_y F(0,\lambda)$: 
\begin{equation} \label{perturbed eigenvalues}
D_y F(0,\lambda) \varphi(\lambda) = z(\lambda) \varphi(\lambda) \quad \text{such that} \quad z(\lambda_0) = i \kappa_0, \quad \varphi(\lambda_0) = \varphi_0
\end{equation}

These eigenvalues $z(\lambda)$ are continuously differentiable with respect to $\lambda$ near $\lambda_0$, and following E. Hopf we assume that
\begin{equation} \label{nonvanishing speed}
Re z' (\lambda_0) \neq 0, \quad \text{where} \quad ' = \frac{\diff}{\diff \lambda}
\end{equation}
and $Re$ denotes "real part". In this sense the eigenvalue $z(\lambda)$ crosses the imaginary axis with "nonvanishing speed", or the exchange of stability of the trivial solution $\{(0, \lambda)\}$ is nondegenerate.

Apart from the spectral properties \eqref{spectral of F} and \eqref{perturbed eigenvalues}, we need more assumptions in order to give the evolution equation \eqref{evolution} a meaning in the (possibly) infinite-dimensional Banach space $Z$. The following condition on the linearization serves this purpose:

\begin{equation} \label{semigroup}
\begin{aligned}
& L_0 = D_y F(0,\lambda_0) \quad \text{as a mapping in} \ Z, \ \text{with dense domain} \\ 
& \text{of definition} \ D(L_0) = X, \ \text{generates an analytic (holomorphic)} \\
& \text{semigroup} \ e^{L_0 t}, \ t \geq 0 \quad \text{on} \ Z \ \text{that is compact for} \ t > 0
\end{aligned}
\end{equation}
Finally we formulate the theorem:

\begin{thm} (Hopf Bifurcation Theorem) \label{Hopf Bifurcation Theroem}
For the parameter-dependent evolution equation \eqref{evolution} in a Banach space $Z$ we make the regularity assumptions \eqref{embedding}, \eqref{def of F}, \eqref{smoothness of F} on the mapping $F$. We make the spectral assumptions \eqref{spectral of F}, \eqref{perturbed eigenvalues}, \eqref{nonvanishing speed} on the linearization $D_y F(0,\lambda)$ along the trivial solutions:
\begin{equation*}
\begin{aligned}
& D_y F(0,\lambda) \varphi(\lambda) = z(\lambda) \varphi(\lambda) \quad \text{with} \quad z(\lambda_0) = i \kappa_0 \neq 0, \ z(\lambda) \ are \\
& \text{simple eigenvalues, and we assume the nondegeneracy} \ Re z'(\lambda_0)  \neq 0
\end{aligned}
\end{equation*}
We impose the nonresonance condition: 
\begin{equation} \label{nonresonance}
\text{for any} \ n \in \mathbb{Z}\backslash \{1, -1\} \ \ i n \kappa_0 \ \text{is not an eigenvalue of} \ L_0 = D_y F(0,\lambda_0)
\end{equation}
We assume that the operator $L_0$ generates a holomorphic semigroup according to \eqref{semigroup} \[ e^{L_0 t} \in L(Z, Z) \ for \ t \geq 0, \ \text{which is compact for} \ t > 0 \]
Then there exists a continuously differentiable curve $\{(y(r), \lambda (r))\}$ of (real) $\dfrac{2 \pi}{\kappa (r)}$ - periodic solutions of \eqref{evolution} through $(y(0), \lambda (0)) = (0, \lambda_0)$ with $\kappa(0) = \kappa_0$ in \\ ${\text{\large{( C}}}^{ \ 1+\alpha}_{{2 \pi} / \kappa (r)}\text{\large{($\mathbb{R}$, Z)}} \ \bigcap \ {\text{\large{C}}}^{ \ \alpha}_{{2 \pi} / \kappa (r)}\text{\large{($\mathbb{R}$, X) )}} \times \text{\large{$\mathbb{R}$}}$. Every other periodic solution of \eqref{evolution} in a neighborhood of $(0, \lambda_0)$ is obtained from $(y(r), \lambda (r))$ by a phase shift $S_\theta y(r)$.
\end{thm}

\begin{remark}
The phase shift operator is defined by $S_\theta y (t) = y (t + \theta)$ and in particular one can obtain: \[ y(-r) = S_{\pi / \kappa (r)} y (r), \ \kappa (-r) = \kappa (r), \ \text{and} \ \lambda (-r) = \lambda (r) \ \text{for all} \ r \in (- \delta, \delta)\]
\end{remark}

\begin{remark}
Above we introduced the Banach spaces of $\frac{2 \pi}{\kappa (r)}$ - periodic H\"{o}lder continuous functions having values in $X$ or $Z$ according to the following definition:

\begin{equation*} \label{Hölder X}
\begin{aligned}
C^{\alpha}_{2 \pi} (\mathbb{R}, X) = \bigl \{ & y : \mathbb{R} \rightarrow X \ / \ y(t + 2 \pi) = y(t), \ t \in \mathbb{R} \\ & ||y||_{X, \alpha} := \max_{t \in \mathbb{R}} ||y(t)||_X + \sup_{s \neq t} \frac{||y(t) - y(s)||_X}{|t - s|^\alpha} < \infty \bigr \}
\end{aligned}
\end{equation*}

\begin{equation*} \label{Hölder Z}
\begin{aligned}
C^{1 + \alpha}_{2 \pi} (\mathbb{R}, Z) = \bigl \{ & y : \mathbb{R} \rightarrow Z \ / \ y, \frac{\diff y}{\diff t} \ (exists) \in C^{\alpha}_{2 \pi} (\mathbb{R}, Z) \\ &||y||_{Z, 1 + \alpha} := ||y||_{Z, \alpha} + ||\frac{\diff y}{\diff t}||_{Z, \alpha}  \bigr \}
\end{aligned}
\end{equation*}
The H\"{o}lder exponent $\alpha$ is in the interval $(0,1]$. Clearly $C^{\alpha}_{2 \pi} (\mathbb{R}, X) \bigcap C^{1 + \alpha}_{2 \pi} (\mathbb{R}, Z)$ is a Banach space with norm $||y||_{X, \alpha} + ||\dfrac{\diff y}{\diff t}||_{Z, \alpha}$ \ (cf. \eqref{embedding}).
\end{remark}

\begin{proof}
For the proof of the theorem we refer to \cite{kiel1}.
\end{proof}
The following formula is useful for estimating the left-hand side of \eqref{nonvanishing speed}.

\begin{prop} \label{nonvanishing speed prop}
With the setting as above it holds that
\begin{equation} \label{nonvanishing speed formula}
z'(\lambda_0) = \langle D^2_{y \lambda} F (0, \lambda_0) \varphi_0, \varphi^*_0 \rangle
\end{equation}
where $\varphi^*_0$ is the eigenvector of the dual operator $L^*_0$ with eigenvalue $i \kappa_0$ such that $\langle \varphi_0, \varphi^*_0 \rangle = 1$.
\end{prop}

\begin{proof}
By assumption \eqref{semigroup}, \ $L_0 : Z \rightarrow Z$ is densely defined, and thus its dual operator $L^*_0 : Z^* \rightarrow Z^*$ exists. Let $\langle  \cdot, \cdot \rangle$ denote the bilinear pairing of $Z$ and $Z^*$, then let us choose the eigenvector $\varphi^*_0$ of $L^*_0$ with eigenvalue $i \kappa_0$ so that $\langle \varphi_0, \varphi^*_0 \rangle = 1$. Differentiation of \eqref{perturbed eigenvalues} with respect to $\lambda$ at $\lambda = \lambda_0$ yields
\begin{align*}
D^2_{y \lambda} F(0, \lambda_0) \varphi_0 + D_y F (0, \lambda_0) \varphi ' (\lambda_0) & = z'(\lambda_0) \varphi_0 + z(\lambda_0) \varphi '(\lambda_0) \\
\langle D^2_{y \lambda} F(0, \lambda_0) \varphi_0, \varphi^*_0 \rangle + \langle L_0 \varphi ' (\lambda_0), \varphi^*_0  \rangle & = \langle z'(\lambda_0) \varphi_0, \varphi^*_0 \rangle + \langle z(\lambda_0) \varphi '(\lambda_0), \varphi^*_0  \rangle \\
\langle D^2_{y \lambda} F(0, \lambda_0) \varphi_0, \varphi^*_0 \rangle + \langle \varphi ' (\lambda_0), L^*_0 \varphi^*_0 \rangle & = z'(\lambda_0) + i \kappa_0 \langle \varphi ' (\lambda_0), \varphi^*_0 \rangle \\
\langle D^2_{y \lambda} F(0, \lambda_0) \varphi_0, \varphi^*_0 \rangle & = z'(\lambda_0)
\end{align*}
Where the last equality holds since $L^*_0 \varphi^*_0 = i \kappa_0 \varphi^*_0$ and $\langle \cdot, \cdot \rangle$ is bilinear.
\end{proof}
The following formula is useful for determining the type of the bifurcation of periodic solutions (for the proof see \cite{kiel1}).
\begin{prop}
Let $\{(y(r),\lambda(r))\}$ be the curve of $\frac{2\pi}{\kappa(r)}$-periodic solutions of \eqref{evolution} according to Theorem~\ref{Hopf Bifurcation Theroem}. Then
\begin{equation} \label{type of bifurcation}
\left. \frac{d^2}{dr^2}\lambda(r) \right|_{r=0}=\frac{1}{Re z'(\lambda_0)} Re D^2_{rr} \Phi^0 \qquad \qquad \text{where}
\end{equation}
\begin{equation} \label{long formula}
\begin{split}
D^2_{rr} \Phi^0 = -& \left\langle D^3_{yyy}F(0,\lambda_0)[\varphi_0,\varphi_0,\bar{\varphi}_0], \varphi_0^* \right\rangle \\
-&\left\langle D^2_{yy}F(0,\lambda_0)[\bar{\varphi_0}, (2i\kappa_0-L_0)^{-1}D^2_{yy}F(0,\lambda_0)[\varphi_0,\varphi_0]], \varphi_0^* \right\rangle \\
+&2\left\langle D^2_{yy}F(0,\lambda_0)[\varphi_0, L_0^{-1}D^2_{yy}F(0,\lambda_0)[\varphi_0,\bar{\varphi}_0]], \varphi_0^* \right\rangle
\end{split}
\end{equation}
where $\varphi_0, \varphi_0^*$ and $\langle\cdot, \cdot\rangle$ are defined as in the proposition above.  
\end{prop}

\subsection{Hopf bifurcation at multiple eigenvalues} \label{Hopf bifurcation at multiple eigenvalues, section}

In this section we follow \cite{kiel2}. Consider the following abstract evolution equation in a Hilbert space $Z$ with norm $||\cdot ||$ and scalar product $(\cdot, \cdot)$
\begin{equation} \label{evolution equation ME}
\frac{\diff y}{\diff t}+Ly+B(\lambda)y = G(\lambda, y)
\end{equation}
Assume $B(0)=0$ and $G(\lambda,0)=0$ so that \eqref{evolution equation ME} has the trivial solution $y=0$ for all $\lambda \in \RR$. We study the bifurcation of the trivial solution into periodic solutions at $\lambda=0$. 

The linear operator $L$ is densely defined and satisfies $L=C_0+B_0$ where $C_0$ is real, self-adjoint, positive definite and $C_0^{-1}$ is compact, and $B_0$ is real, and
\begin{equation} \label{assumption on B_0}
||B_0u|| \leq c_1 ||C_0^\alpha u|| \qquad \text{for} ~ u \in D(C_0^\alpha), \quad 0\leq \alpha < 1
\end{equation}  
This implies that $B_0$ is $C_0$-bounded with relative bound $0$ (see \cite{kato}, p. 190). The real operators $B(\lambda)$ and $G(\lambda, \cdot)$ depend analytically on $u$ and the real parameter $\lambda$, and in particular
\begin{equation} \label{analytic expansion of B and F}
\begin{split}
B(\lambda)u &=\lambda Bu + \lambda^2 B_2u+... \\
G(\lambda, u) &= G(u) + \lambda G_1(u)+...
\end{split}
\end{equation}
where $G$ and $G_j$ are "homogeneous polynomials" of order $k$ and $k_j$ (see \cite{hille}, Chap. 26). We assume finally that
\begin{equation} \label{final assumption}
\begin{split}
&||Bu|| \leq c_2 ||C_0^\alpha u||, \quad ||B_ju|| \leq c_{3,j} ||C_0^\alpha u|| \qquad \text{for} ~ 0\leq \alpha < 1 , \ j=2,3,... \\
&||Gu|| \leq c_4 ||C_0^\beta u||^k, \quad ||G_ju|| \leq c_{5,j} ||C_0^{\beta_j} u||^k_j \qquad \text{when}
\end{split}
\end{equation}
\begin{equation*}
0\leq \beta \leq \frac{k+1}{2k}, \quad 0\leq \beta_j \leq \frac{k_j+1}{2k_j}, \quad 2\leq k \leq k_j, \quad j=1,2,...
\end{equation*}
We shall only consider the case $k=2$. We put the following restriction on the constants
\begin{equation} \label{series assumption}
\begin{split}
&\text{the infinite series} \sum_{j=2}^\infty c_{3,j}\lambda^j \quad \text{and} \quad \sum_{j=1}^\infty c_{5,j}\lambda^j \quad \text{have a common} \\
&\text{nontrivial radius of convergence} \quad r_0>0
\end{split}
\end{equation} 
As before we assume
\begin{equation} \label{assumption on A ME}
\begin{split}
& i\mu_0, \mu_0>0 \ \text{is an eigenvalue of $L$, but there is no eigenvalue of $L$ of the} \\
& \text{form} ~ in\mu_0, n\in \ZZ \backslash \{-1,1\}
\end{split} 
\end{equation}
We assume in addition that
\begin{equation} \label{semisimple eigenvalue}
i\mu_0 \ \text{is a semisimple eigenvalue of multiplicity} ~ r\geq 1 
\end{equation}
Letting $L^*$ denote the adjoint operator, we choose the following dual bases for the kernels:
\begin{equation} \label{dual bases}
\begin{split}
\{\varphi_1,..., \varphi_r\} \subset N(L-i\mu_0),& \quad \{\varphi_1^*,..., \varphi_r^*\} \subset N(L^*+i\mu_0) \\
(\varphi_k, \varphi_l^*)&=\frac{1}{2\pi} \delta_{kl}
\end{split} 
\end{equation}
The period $2\pi/\omega$ of the bifurcating solution is a priori unknown. We substitute $t/\omega$ for $t$ to obtain
\begin{equation*}
\omega \frac{\diff \tilde{y}}{\diff t}+L\tilde{y}+B(\lambda)\tilde{y}=G(\lambda, \tilde{y}) \qquad \text{for} ~ \tilde{y}(t)=y(\frac{t}{\omega})
\end{equation*}
Replacing $\tilde{y}$ by $y$, the problem now is to prove the existence of nontrivial solutions of the equation
\begin{equation} \label{bifurcation equation ME}
\omega \frac{\diff y}{\diff t} + Ly+B(\lambda)y=G(\lambda,y), \qquad y(0)=y(2\pi)
\end{equation}
which issue from the trivial zero solution. 

To this end we introduce Hilbert spaces: $H_0=L_2[(0,2\pi),Z]$ with scalar product $(\cdot, \cdot)_0=\int_0^{2\pi}(\cdot, \cdot) \diff t$ and $H_2=cl_{||\cdot||_2} \left\lbrace y \in H_0 \ / \ \frac{dy}{dt} \in H_0, \ y \in L_2[(0,2\pi),D(C_0)], \ y(0)=y(2\pi)] \right\rbrace$ with scalar product $(\cdot, \cdot)_2=\int_0^{2\pi}\mu_0^2 (\frac{d}{dt}\cdot, \frac{d}{dt}\cdot) + (C_0\cdot, C_0\cdot) \diff t$.

Consider the operator $J_0=\mu_0 \frac{d}{dt}+L$ with $D(J_0)=H_2$. Let $P_0: H_0 \rightarrow N(J_0)$ be the projector defined by $P_0y=\sum_{|l|=1}^r (y, \psi_l^*)_0 \psi_l$ where
\begin{equation}
\psi_l=
\begin{cases}
e^{-it} \varphi_l \\
e^{it} \bar{\varphi}_{-l}
\end{cases} \quad
\psi_l^*=
\begin{cases}
e^{-it} \varphi_l^* \qquad l=1,...,r\\
e^{it} \bar{\varphi}_{-l}^* \qquad l=-1,...,-r
\end{cases}
\end{equation}
And let $Q_0=I-P_0$. The operator $G$ is a continuous homogeneous polynomial from $H_2$ into $H_0$, generated by a symmetric polar form (see \cite{hille}, Chap. 26): $G(y)=G^{(2)}(y,y)$.

Now we introduce the abbreviation $d/dt=\Upsilon$ and $\omega=\mu_0+\mu$ and write \eqref{bifurcation equation ME} as
\begin{equation} \label{bifurcation equation ME rewritten}
J_0y + \mu \Upsilon y + B(\lambda)y= G(\lambda, y), \qquad y \in H_2
\end{equation}
The projectors $P_0$ and $Q_0$ commute with $\Upsilon$ as well as with $J_0$. We write $y=P_0y+Q_0y=v+w$ whence \eqref{bifurcation equation ME rewritten} becomes
\begin{equation} \label{bifurcation equation Lyapunov-Schmidt}
\begin{split}
\mu \Upsilon v+P_0B(\lambda)(v+w)&=P_0G(\lambda,v+w) \\
J_0w+\mu \Upsilon w+Q_0B(\lambda)(v+w)&=Q_0G(\lambda,v+w)
\end{split}
\end{equation}
This decomposition into an infinite-dimensional equation $(b)$ and a finite-dimensional "bifurcation equation" $(a)$ is the Lyapunov-Schmidt decomposition for the evolution equation \eqref{evolution equation ME}. The solution $w=w(\mu,\lambda,v)$ of \eqref{bifurcation equation Lyapunov-Schmidt}$(b)$, when put into equation \eqref{bifurcation equation Lyapunov-Schmidt}$(a)$, yields
\begin{equation} \label{bifurcation equation long}
\begin{split}
\mu \Upsilon v +\lambda P_0Bv+\sum_{j=2}^\infty \lambda^j P_0B_jv = &2P_0G^{(2)}(v, J_0^{-1}G^{(2)}(v,v))- \lambda^2 \sum_{j=1}^\infty \sum_{n,m=0}^\infty \lambda^{j+n-1} \mu^m P_0B_jB_{nm}v + \\
& + O(||v||_0^5)+O(\mu ||v||_0^3)+ O(\lambda ||v||_0^3) \qquad ~~ (B_1=B)
\end{split}
\end{equation}
Where we write the terms of third order as $E^{(3)}(v)$. Next we rewrite \eqref{bifurcation equation long} in coordinates relative to the basis $\{\psi_l\}$: $v=\sum_{|l|=1}^r h_l \psi_l, \quad v\leftrightarrow (h_1,...,h_r, h_{-1},...,h_{-r})^t$,
\begin{equation} \label{corresponding matrices}
\Upsilon\leftrightarrow
\begin{pmatrix}
-i \\
& \ddots \\
& & -i \\
& & & i \\
& & & & \ddots \\
& & & & & i
\end{pmatrix} \qquad
P_0B \leftrightarrow 2\pi
\begin{pmatrix}
(B\varphi_k, \varphi_l^*) & 0 \\
0 & (B\bar{\varphi}_k, \bar{\varphi}_l^*)
\end{pmatrix}_{k,l=1,...,r}
\end{equation}
\begin{equation} \label{formula for E^3}
E^{(3)}(v)\leftrightarrow 2 \left( \sum_{i,j,k=1}^r a_{ijk}^l h_ih_jh_{-k}, \sum_{i,j,k=1}^r \overline{a_{ijk}^l} h_{-i}h_{-j}h_{k} \right)_{l=1,...,r}^t \qquad \text{where}
\end{equation}
\begin{equation*}
a_{ijk}^l=4\pi \left( G^{(2)}(L^{-1}G^{(2)}(\varphi_j, \bar{\varphi}_k), \varphi_i), \varphi_l^* \right) +2\pi \left( G^{(2)}((L-2i\mu_0)^{-1}G^{(2)}(\varphi_j, \varphi_i), \bar{\varphi}_k), \varphi_l^* \right)
\end{equation*}
Of course in $(L-2i\kappa_0)$ \ $i$ denotes the imaginary unit and not the index variable. For physical reasons, all operators are assumed to be real, hence we have $h_{-l}=\bar{h}_l$. Since the last $r$ equations of \eqref{bifurcation equation long} written in coordinates are conjugate to the first $r$ equations if $v$ is real, we only consider the first $r$ complex equations and write them as $2r$ real equations. The function $v$ then depends on $2r$ real variables: $h_l=x_l+iy_l, \quad v\leftrightarrow (x_1,y_1,...,x_r,y_r)^t \in \RR^{2r}$. We thus get a real system in $\RR^{2r}$ with $2r+2$ variables. In the following we identify $v$ with all phase shifted functions and look only for a special representative, namely that with $y_j=0$ for some fixed $j$. In this case the number of variables is reduced to $2r+1$.

\begin{thm}(Hopf Bifurcation at multiple eigenvalues) \label{Hopf Bifurcation at multiple eigenvalues}
Given the setting and notations introduced above, let $v^0 \in \RR^{2r}$ and $\rho^0 \in \RR$ be a solution of
\begin{equation} \label{necessary equation multiple eigenvalue}
\Upsilon v + \rho P_0Bv = E^{(3)}(v)
\end{equation}
where $v_{2j}^0=y_j^0=0$ and $v^0 \neq 0$. Then there exists a nontrivial solution curve $(\mu,\lambda,v)$ of \eqref{bifurcation equation long} passing through $(0,0,0)$ and emanating in the direction $v^0$, provided \begin{equation} \label{existence condition}
det \left( \hat{P_0B}v^0 \quad \hat{\Upsilon}-D_v\hat{E}^{(3)}(P_{2j}v^0) \right) \neq 0
\end{equation}
\end{thm}

\begin{remark}
In \eqref{existence condition} $P_k$ denotes the projection of $\RR^{2r}$ onto $\RR^{2r-1}$ which deletes the $k$-th coordinate. Also $\hat{E}^{(3)}(x_1,y_1,...,x_j,x_{j+1},...,x_r,y_r)=E^{(3)}(x_1,y_1,...,x_j,0,x_{j+1},...,x_r,y_r)$ and the $2r \times (2r-1)$-matrices $\hat{\Upsilon}, \hat{P_0B}$ result from $\Upsilon, P_0B$ by omitting the $2j$-th column. Finally $D_v$ denotes differentiation w.r.t $v \in \RR^{2r-1}$.
\end{remark}

\begin{proof}
For the proof we refer to \cite{kiel2}.
\end{proof}

\section{Proof of the main results: Nonsymmetric model} \label{nonsymmetric model}
In this section we are in the setting of the Theorem \ref{main theorem, nonsymmetric}, so in particular $M=M_{ns}$. And we aim to show that the Hopf Bifurcation Theorem (cf. Theorem \ref{Hopf Bifurcation Theroem}) applies to our model. Clearly the conditions \eqref{embedding} - \eqref{smoothness of F} are satisfied in our case. Now we note some useful properties of the spaces $X$ and $Z$.

\begin{prop} \label{properties of X and Z}
Let $X$ and $Z$ be given by \eqref{our X, Z} then
\begin{enumerate}
\item[\textit{(i)}] $X \subset Z$ is dense (w.r.t. $L^2$ norm)

\item[\textit{(ii)}] $Z$ is a closed subspace of $[L^2(0,1)]^4$ and in particular $X$ is compactly embedded in $Z$: $X \hookrightarrow_C Z$
\end{enumerate}
\end{prop}

\begin{proof}
\begin{enumerate}
\item[\textit{(i)}] Take any $\varphi=(\varphi_1,...,\varphi_4) \in Z $ then in particular $\varphi_j \in L^2(0,1)$ hence by the usual approximation argument
\begin{equation*}
\exists ~ \{ f_n^{(j)} \}_{n=1}^{\infty} \subset C_0^{\infty}(0,1) \quad \text{s.t.} \quad f_n^{(j)} \rightarrow \varphi_j \quad \text{in} ~ L^2 \quad \text{as} ~ n\rightarrow \infty, \quad \text{for} ~ j=1,...,4
\end{equation*}
By the Cauchy-Schwarz inequality we obtain
\begin{equation*}
a_n := \sum_{j=1}^4 \int_0^1 f_n^{(j)} \diff x \longrightarrow \sum_{j=1}^4 \int_0^1 \varphi_j \diff x = 0 \qquad \text{as} ~ n\rightarrow \infty
\end{equation*}
We now consider a new approximating sequence given by
\begin{equation*}
g_n^{(j)} = f_n^{(j)} - \dfrac{h}{4} a_n \qquad \forall n \in \mathbb{N}, \quad j=1,...,4
\end{equation*}
where $h \in C_0^\infty (0,1)$ is nonnegative and has total mass equal to $1$. We introduced the cut-off function $h$ to obtain $g_n^{(j)} \in C_0^\infty (0,1)$. Now note that
\begin{equation*}
\sum_{j=1}^4 \int_0^1 g_n^{(j)} \diff x = \sum_{j=1}^4 \int_0^1 f_n^{(j)} \diff x - \sum_{j=1}^4 \dfrac{a_n}{4} \int_0^1 h \diff x = a_n - a_n = 0
\end{equation*}
So we see that $g_n := (g_n^{(1)},...,g_n^{(4)}) \in X$ and remains to show the approximating property of this sequence:
\begin{equation*}
\begin{split}
||g_n - \varphi||_{L^2}^2 & = \sum_{j=1}^4 \int_0^1 |g_n^{(j)} - \varphi_j |^2 \diff x = \sum_{j=1}^4 \int_0^1 |f_n^{(j)} - \varphi_j - \dfrac{h}{4} a_n|^2 \diff x \lesssim \\
& \lesssim \sum_{j=1}^4 \int_0^1 |f_n^{(j)} - \varphi_j |^2 \diff x + \sum_{j=1}^4 \dfrac{a_n^2}{16} \int_0^1 h^2 \diff x = \\
& = ||f_n - \varphi||_{L^2}^2 + ||h||_{L^2}^2 \dfrac{a_n^2}{2} \rightarrow 0 \qquad \text{as} ~ n\rightarrow \infty
\end{split}
\end{equation*}
where we used the notation $f_n=(f_n^{(1)},...,f_n^{(4)})$ and the fact that $a_n \rightarrow 0$.

\item[\textit{(ii)}] Take any $\{\varphi_n\}_{n=1}^\infty \subset Z$ with $\varphi_n \rightarrow \varphi$ in $L^2$, then by the Cauchy-Schwarz inequality
\begin{equation*}
\sum_{j=1}^4 \int_0^1 \varphi_j \diff x = \lim_{n\rightarrow \infty} \sum_{j=1}^4 \int_0^1 \varphi_n^{(j)} \diff x = 0
\end{equation*}
hence $\varphi \in Z$. Now to prove $X \hookrightarrow_C Z$ take any sequence $\{\varphi_n\}_{n=1}^\infty \subset X$ with $||\varphi_n||_{H^2} \lesssim 1$, then it admits a convergent subsequence (which we don't relabel) in $[L^2(0,1)]^4$ since $[H^2(0,1)]^4 \hookrightarrow_C [L^2(0,1)]^4$ :
\begin{equation*}
\exists ~ \varphi \in [L^2(0,1)]^4 \quad \text{s.t.} ~ \varphi_n \rightarrow \varphi \quad \text{in} ~ L^2
\end{equation*}
But since $\{\varphi_n\}_{n=1}^\infty \subset Z$ and $Z$ is closed we obtain $\varphi \in Z$. Thus any bounded sequence in $X$ admits a convergent subsequence in $Z$. 
\end{enumerate}
\end{proof}

Next we prove some properties of matrices resulting as Fourier coefficients after application of corresponding operators, which would be useful in proving that the latter are Fredholm operators.

\begin{prop} \label{properties of coefficient matrices}
With the above-mentioned setting and the notation of \eqref{eigenvalue problem 2v} the following hold true:
\begin{enumerate}
\item[\textit{(i)}]
\begin{equation} \label{invertability of matrices}
- \widetilde{M}(m, \lambda_0) + i \kappa_0  \quad \text{is invertible for any} ~ 1 \neq m \in \ZZ
\end{equation}

\item[\textit{(ii)}] $\exists ~ m_0 = m_0 (\lambda_0, \kappa_0, \varepsilon)>1$ ~ s.t. for any $|m| > m_0$
\begin{equation} \label{norm bound for matrices}
\left|\left| \left(- \widetilde{M}(m, \lambda_0) + i \kappa_0 \right)^{-1} \right|\right|_\infty \lesssim \dfrac{1}{m^2}
\end{equation}

\end{enumerate}
\end{prop}

\begin{proof}
\begin{enumerate}
\item[\textit{(i)}] For $m=1$ we obtain the matrix $( - \widetilde{M}(1, \lambda_0) + i \kappa_0 )$ which is not invertible by the definition of $\kappa_0$ (cf. \eqref{lambda0}). We want to show that $i \kappa_0$ is not an eigenvalue of $\widetilde{M}(m, \lambda_0)$ for $m \neq 1$, i.e. 
\begin{equation*}
z_j(\dfrac{2 \pi m}{\lambda_0}) \neq i \kappa_0 \qquad \text{for} ~ j=1,...,4 ~ \text{and} ~ m \neq 1
\end{equation*}
where $z_j$ are defined in \eqref{solutions of eigenvalue problem 1v}. But the plots of functions $z_j$ show that only one of them crosses the imaginary axis at $\kappa_0$ and only once, namely when $m=1$: $z_1(\frac{2 \pi}{\lambda_0})=i \kappa_0$.

\item[\textit{(ii)}] We start by pulling out the desired term from the inverse
\begin{equation*}
\begin{split}
\left(- \widetilde{M}(m, \lambda_0) + i \kappa_0 \right)^{-1} &= \dfrac{\lambda_0^2}{\varepsilon 4 \pi^2 m^2} \left( Id + \dfrac{\lambda_0 i}{\varepsilon 2 \pi m} U + \dfrac{\lambda_0^2}{\varepsilon 4 \pi^2 m^2} (DA + i \kappa_0) \right)^{-1} =:\\
& =: \dfrac{\lambda_0^2}{\varepsilon 4 \pi^2 m^2} (Id - T_m)^{-1}
\end{split}
\end{equation*}
Since
\begin{equation*}
||T_m||_\infty \leq \frac{|\lambda_0|}{\varepsilon 2 \pi} ||U||_\infty \cdot \frac{1}{|m|} + \dfrac{\lambda_0^2}{\varepsilon 4 \pi^2} (|\kappa_0| + ||DA||_\infty) \cdot \dfrac{1}{m^2} 
\end{equation*}
we see that $\exists ~ m_0 = m_0 (\lambda_0, \kappa_0, \varepsilon)>1$ ~ s.t. ~ $||T_m||_\infty \leq \frac{1}{2}< 1 \qquad \forall \ |m|>m_0$. So for any fixed $m \in \ZZ$ with $|m|>m_0$ we apply Neumann's theorem to obtain
\begin{equation*}
(Id - T_m)^{-1} = \sum_{j=0}^\infty T_m^j
\end{equation*}
this in particular implies $\exists ~ n_0 \in \mathbb{N} ~ \text{s.t.} ~ \forall \ k \geq n_0$ \quad $\left|\left|(Id - T_m)^{-1} - \sum_{j=0}^k T_m^j \right|\right|_\infty < 1$ hence
\begin{equation*}
\left|\left|(Id - T_m)^{-1}\right|\right|_\infty < 1 + \sum_{j=0}^k ||T_m||^j_\infty < 1 + \sum_{j=0}^\infty \dfrac{1}{2^j} = 3 
\end{equation*}
Putting things together we see
\begin{equation*}
\left|\left| \left(- \widetilde{M}(m, \lambda_0) + i \kappa_0 \right)^{-1} \right|\right|_\infty \leq \dfrac{3 \lambda_0^2}{\varepsilon 4 \pi^2 m^2} \lesssim \dfrac{1}{m^2} \qquad \forall \ |m|>m_0
\end{equation*}
\end{enumerate}
\end{proof}

\begin{prop}
Given the setting mentioned above the operators $\pm i \kappa_0 - D_yF(0, \lambda_0) $ are Fredholm operators from $X$ to $Z$, of index $0$.
\end{prop}

\begin{proof}
W.l.o.g. we prove the claim for $B:=i \kappa_0 - D_yF(0, \lambda_0) $. Clearly $B \in L(X,Z)$, so we start by showing that its range is close in $Z$. Take $\{g_n\}_{n \geq 1} \subset X$ s.t. $Bg_n\rightarrow f$ and let us show that $f \in \Re (B)$. Expanding $g_n$ into Fourier series and noting that $B=i \kappa_0+ \frac{1}{\lambda_0}U \partial_x+DA-\frac{\varepsilon}{\lambda_0^2} \partial_x^2$ we see that
\begin{equation} \label{Fourier expansion of B}
\begin{split}
g_n(x)&= \sum_{m \in \ZZ} \hat{g}_n(m) e^{i2 \pi mx} \quad \text{where} \quad \hat{g}_n(m)= \int_0^1 g_n(x)e^{-i2 \pi mx} \diff x \quad \in \CC^4 \\
Bg_n (x) &= \sum_{m \in \ZZ} \left(- \widetilde{M}(m, \lambda_0) + i \kappa_0 \right)\hat{g}_n(m) e^{i2 \pi mx} =:\sum_{m \in \ZZ} \hat{w}_n(m) e^{i2 \pi mx}
\end{split}
\end{equation}
where we used the notation $w_n:=Bg_n$. Our goal would be to construct a new sequence $\{\widetilde{g}_n\}$, by eliminating $\{g_n\}$ in the direction of the $0$ eigenvalue of $(- \widetilde{M}(1, \lambda_0) + i \kappa_0 ) =: M_0$, satisfying
\begin{enumerate}
\item[(a)] $B\widetilde{g}_n = Bg_n$ \qquad $\forall n \in \mathbb{N}$
\item[(b)] $\{\widetilde{g}_n\}$ is Cauchy in $X$
\end{enumerate}
Using numerical simulations (for convenience we take $\varepsilon=0.1$, for other values of $\varepsilon$ only the numerical calculations given below change) one obtains $k_0\approx -4.47675$ and $\kappa_0 \approx -4.54605$, but then the eigenvalues of $M_0$ are
\begin{equation} \label{eigenvalues of M_0}
\lambda_1 \approx 1.99739 - 13.5171 i, \quad \lambda_2 \approx 2.00414 - 9.02281 i \quad \lambda_3 \approx 2.01502 + 4.35574 i \quad \lambda_4 \approx 0 
\end{equation} 
Since all four eigenvalues are distinct, the matrix is diagonalizable, i.e. there exists a basis of $\CC^4$ consisting of the eigenvectors of $M_0$ corresponding to the above eigenvalues, which we denote by $v_1,v_2,v_3,v_4 \in \CC^4$. Now we expand the vectors $\hat{g}_n(1), \ n \in \ZZ$ in this basis: $\exists \ c_n^1, c_n^2, c_n^3, c_n^4 \in \CC$ s.t. $\hat{g}_n(1) = \sum_{j=1}^4 c_n^j v_j$ and eliminate $g_n$ in the direction $v_4$:
\begin{equation}
\widetilde{g}_n(x) := \sum_{m \neq 1} \hat{g}_n(m) e^{i2 \pi mx} + (c_n^1v_1 + c_n^2v_2+c_n^3v_3) e^{i2 \pi x}  
\end{equation}
Clearly (a) holds true, since $B(v_4e^{i2 \pi x})=M_0v_4e^{i2 \pi x}=0$. Next we show that the sequences of coefficients $c_n^j$ are in fact Cauchy for $j=1,2,3$. Since $\{Bg_n\}$ is Cauchy in $[L^2]^4$
\begin{equation*}
\begin{split}
|M_0(\hat{g}_k(1)-\hat{g}_l(1))| \leq & \left(\sum_{m \in \ZZ} \left|\left(- \widetilde{M}(m, \lambda_0) + i \kappa_0 \right)(\hat{g}_k(m)-\hat{g}_l(m))\right|^2\right)^{1/2} = \\
= & ||Bg_k-Bg_l||_{L^2} \longrightarrow 0 \qquad \text{as} ~ k,l\rightarrow \infty
\end{split}
\end{equation*}
So we obtain
\begin{equation*}
M_0(\hat{g}_k(1)-\hat{g}_l(1))=\lambda_1 (c_k^1-c_l^1) v_1+\lambda_2 (c_k^2-c_l^2) v_2 + \lambda_3 (c_k^3-c_l^3) v_3 \longrightarrow 0 \qquad \text{as} ~ k,l\rightarrow \infty
\end{equation*}
Let $\langle v_1^*,...,v_4^* \rangle$ be the basis of $(\CC^4)^*$ dual to $\langle v_1,...,v_4 \rangle$, then applying $v_1^*,v_2^*,v_3^*$ in the above relation we deduce
\begin{equation} \label{coefficients are Cauchy}
c_k^j-c_l^j\longrightarrow 0 \quad \text{as} ~ k,l\rightarrow \infty \qquad \text{for} ~ j=1,2,3
\end{equation}
We then can estimate, using $1 \leq m^2 \leq m^4$ for $m \in \ZZ$
\begin{equation} \label{H2 norm of the difference}
\begin{split}
||\widetilde{g}_k-\widetilde{g}_l||_{H^2}^2 = &||\widetilde{g}_k-\widetilde{g}_l||_{L^2}^2+||\widetilde{g}'_k-\widetilde{g}'_l||_{L^2}^2+||\widetilde{g}''_k-\widetilde{g}''_l||_{L^2}^2 \lesssim \\
\lesssim & \sum_{m \neq 1} m^4 |\hat{g}_k(m)-\hat{g}_l(m)|^2 + |(c_k^1-c_l^1) v_1+(c_k^2-c_l^2) v_2 + (c_k^3-c_l^3) v_3|^2
\end{split}
\end{equation}
From \eqref{coefficients are Cauchy} we see that the second term tends to $0$ as $k,l\rightarrow \infty$, so to conclude (b) remains to show the latter property also for the first term. This is where we invoke Proposition ~\ref{properties of coefficient matrices}, firstly
\begin{equation*}
\hat{g}_n(m) = \left(- \widetilde{M}(m, \lambda_0) + i \kappa_0 \right)^{-1} \hat{w}_n(m) \qquad m \neq 1, n \in \mathbb{N}
\end{equation*}
then using \eqref{norm bound for matrices} we get
\begin{equation*}
\begin{split}
\sum_{m \neq 1} m^4 |\hat{g}_k(m)-\hat{g}_l(m)|^2 =&  \sum_{m \neq 1} m^4 \left|\left(- \widetilde{M}(m, \lambda_0) + i \kappa_0 \right)^{-1}(\hat{w}_k(m)-\hat{w}_l(m))\right|^2 = \sum_{\substack{|m|\leq m_0 \\ m \neq 1}} + \sum_{|m|> m_0} \lesssim \\
\lesssim m_0^4 \max_{\substack{|j|\leq m_0 \\ j \neq 1}} ||( - \widetilde{M}(j, \lambda_0) +& i \kappa_0 ) ^{-1}|| \sum_{\substack{|m|\leq m_0 \\ m \neq 1}} |\hat{w}_k(m)-\hat{w}_l(m)|^2 + \sum_{|m|> m_0} m^4 \dfrac{1}{m^4} |\hat{w}_k(m)-\hat{w}_l(m)|^2 \lesssim  \\
\lesssim&  \sum_{m \neq 1} |\hat{w}_k(m)-\hat{w}_l(m)|^2 \leq ||w_k-w_l||_{L^2}^2\longrightarrow 0 \qquad \text{as} ~ k,l\rightarrow \infty
\end{split}
\end{equation*}
where the last part follows from $w_n=Bg_n$. Note that $g_n \in X$ implies that $\widetilde{g}_n \in X$ since the $0$-th Fourier mode (i.e. integral over $(0,1)$) for $g_n$ and $\widetilde{g}_n$ is the same.

Thus $\{\widetilde{g}_n\}$ is Cauchy in X, hence converges there: $\exists \ g \in X$ s.t. $\widetilde{g}_n \rightarrow g$ in $X$, which implies $B\widetilde{g}_n \rightarrow Bg$ in $Z$. Now recalling (a) we obtain $f = Bg$.

To conclude the Fredholm property we prove $dim(\ker B) = codim (\Re(B))=1$. Take any $\varphi \in \ker B$, expand it into its Fourier series and use \eqref{Fourier expansion of B} to get that $B \varphi =0$ implies
\begin{equation*}
\left(- \widetilde{M}(m, \lambda_0) + i \kappa_0 \right) \hat{\varphi}(m) = 0 \qquad \forall m \in \ZZ
\end{equation*}
Taking into account Proposition ~\ref{properties of coefficient matrices} we see
\begin{equation*}
\hat{\varphi}(m) = 0 \quad \forall m \neq 1 \quad \text{and} \quad \hat{\varphi}(1) \in \ker M_0 = span\{v_4\}
\end{equation*}
this implies $\varphi(x)=c v_4 e^{i2 \pi x}$ with $c \in \CC$ which in turn proves that $\ker B$ is $1$-dimensional
\begin{equation}
\ker B = \{c v_4 e^{i2 \pi x} / c \in \CC\}
\end{equation}
Now pick any $g \in Z$, by means of a direct sum we would like to represent $g$ as $g = \widetilde{g}+h$ where $\widetilde{g} \in \Re{(B)}$ and $h \in Y_0$ for some one-dimensional subspace $Y_0$. So we look for a representation $g = Bf + h$ with $f \in D(B)=X$ and $h \in Y_0$, or equivalently we translate this into representation for Fourier coefficients:
\begin{equation} \label{representation for Fourier coefficients}
\hat{g}(n) = \left(- \widetilde{M}(n, \lambda_0) + i \kappa_0 \right)\hat{f}(n) + \hat{h}(n) \qquad  \quad \forall n\in \ZZ
\end{equation}
Since we are looking for a unique representation we try to "invert" $\hat{g}(n) = \left(- \widetilde{M}(n, \lambda_0) + i \kappa_0 \right)\hat{f}(n)$ as much as possible, and then assign the remaining part to $\hat{h}(n)$. Invoking Proposition~\ref{properties of coefficient matrices} define
\begin{equation} \label{definition of f_hat (n)}
\hat{f}(n):=\left(- \widetilde{M}(n, \lambda_0) + i \kappa_0 \right)^{-1}\hat{g}(n) \qquad \quad \text{for} ~ n \neq 1
\end{equation}
Expand $\hat{g}(1)$ and $\hat{f}(1)$ in the basis of $\CC^4$ consisting of eigenvectors of $M_0$:
\begin{equation*}
\begin{split}
\hat{g}(1) & = \alpha_1 v_1+\alpha_2 v_2+\alpha_3 v_3+\alpha_4 v_4 \qquad \quad \text{with} ~ \alpha_j \in \CC \\
\hat{f}(1) & = \beta_1 v_1+\beta_2 v_2+\beta_3 v_3+\beta_4 v_4 \\
M_0 \hat{f}(1) & = \beta_1 \lambda_1 v_1+\beta_2 \lambda_2 v_2+\beta_3 \lambda_3 v_3
\end{split}
\end{equation*}
Hence we see that in the equality "$\hat{g}(1)=M_0 \hat{f}(1)$" the term $\alpha_4 v_4$ is extra. This suggests to define
\begin{equation} \label{definition of first mode}
\begin{split}
\hat{f}(1) &:=\beta_1 v_1+\beta_2 v_2+\beta_3 v_3 \qquad \quad \text{where} ~ \beta_j = \dfrac{\alpha_j}{\lambda_j} \\
\hat{h}(1) &:= \alpha_4 v_4 \\
\hat{h}(n) &:= 0 \hspace{45mm} \text{for} ~ n \neq 1
\end{split}
\end{equation}
The coefficients $\alpha_j$ are known since $g$ is given. The definitions \eqref{definition of f_hat (n)} and \eqref{definition of first mode} clearly imply \eqref{representation for Fourier coefficients}. Let us now check that $f \in D(B)$, firstly by Proposition~\ref{properties of coefficient matrices}
\begin{equation*}
\begin{split}
\sum_{n \in \ZZ} \left(n^2 |\hat{f}(n)| \right)^2 = & \sum_{|n| \leq m_0} n^4 |\hat{f}(n)|^2 + \sum_{|n| > m_0} n^4 \left|\left(- \widetilde{M}(n, \lambda_0) + i \kappa_0 \right)^{-1}\hat{g}(n)\right|^2 \lesssim \\
\lesssim & \sum_{|n| \leq m_0} + \sum_{|n| > m_0} |\hat{g}(n)|^2 < \infty \qquad \quad \text{since} ~ g \in Z 
\end{split}
\end{equation*}
This shows $f \in [H^2_{per}(0,1)]^4$, but on the other hand $\hat{g}(0)=(DA+ i \kappa_0) \hat{f}(0)$ and by the definition of $Z$ and Proposition~\ref{properties of D}
\begin{equation*}
0 = \sum_{j=1}^4 [\hat{g}(0)]_j = i \kappa_0 \sum_{j=1}^4 [\hat{f}(0)]_j + b \cdot DA \hat{f}(0) = i \kappa_0 \sum_{j=1}^4 [\hat{f}(0)]_j 
\end{equation*}
which implies $\sum_{j=1}^4 [\hat{f}(0)]_j=0$ and therefore $f \in D(B)$.
Thus, we have shown that $Z = \Re(B)+Y_0$ where
\begin{equation*}
Y_0 = span\{v_4 e^{i 2\pi x}\} = \{cv_4 e^{i 2\pi x} / c \in \CC \}
\end{equation*}
which is a one-dimensional subspace of $Z$. So remains to check that the above sum is direct. Take any $\widetilde{g} \in \Re(B)$ and $h \in Y_0$, suppose $\widetilde{g} = Bf$ for some $f \in D(B)$, and let us prove $Bf + h= 0 \Rightarrow Bf = 0$ and $h=0$. Note
\begin{equation*}
0=Bf+h=\sum_{n \neq 1} \left(- \widetilde{M}(n, \lambda_0) + i \kappa_0 \right) \hat{f}(n) e^{i2\pi n x} + (M_0 \hat{f}(1)+cv_4)e^{i2\pi x}
\end{equation*}
implies $\hat{f}(n)=0$ for $n \neq 1$, and $M_0 \hat{f}(1)+cv_4=0$. Now expand $\hat{f}(1)=\sum_{j=1}^4 \beta_j v_j$ then
\begin{equation*}
\beta_1 \lambda_1 v_1+\beta_2 \lambda_2 v_2+\beta_3 \lambda_3 v_3 + c v_4 =0
\end{equation*}
which implies $\beta_1=\beta_2=\beta_3=c=0$, therefore $h=0$ and $\hat{f}(1)=\beta_4 v_4$ but then $Bf=M_0(\beta_4 v_4)e^{i2\pi x}=0$ which proves that $Z= \Re(B)\bigoplus Y_0$.
\end{proof}

\begin{prop} \label{simple eigenvalue}
Given the setting mentioned above, $i \kappa_0$ is a simple eigenvalue of $D_yF(0, \lambda_0)$ with eigenvector $\varphi_0 \notin \Re\left(i\kappa_0 -D_yF(0, \lambda_0)\right)$.
\end{prop}

\begin{proof}
In the Proposition~\ref{properties of coefficient matrices} we saw that only one of the eigenvalues $z_j$ crosses the imaginary axis at $\kappa_0$ and only once, this together with the Proposition~\ref{degeneracy caused by the symmetry} and equality \eqref{eigenvalues of M_0} imply the simplicity of the eigenvalue $i\kappa_0$ for $D_yF(0, \lambda_0)$. Moreover $\varphi_0=e^{i2\pi x}v_0$ with $\widetilde{M}(1, \lambda_0)v_0=i\kappa_0v_0$ (cf. \eqref{eigenvalue i kappa_0 with eigenvector phi_0}). Or equivalently $M_0v_0=0$, which implies $v_0=cv_4$ for some $c \in \CC$ (cf. \eqref{eigenvalues of M_0}).

Now assume the contrary, i.e. $\varphi_0=Bf$ for some $f \in D(B)$, then
\begin{equation*}
cv_4e^{i2\pi x}=\sum_{n \in \ZZ} \left(- \widetilde{M}(n, \lambda_0) + i \kappa_0 \right) \hat{f}(n)e^{i2\pi n x} 
\end{equation*}
equating Fourier coefficients we obtain $\hat{f}(n)=0 \quad \forall n\neq 1$ and $cv_4=M_0\hat{f}(1)$. Expand $\hat{f}(1)$ in the basis of $\CC^4$ formed from the eigenvectors of $M_0$: $\hat{f}(1)=\sum_{j=1}^4 \beta_j v_j$ then we get
\begin{equation*}
cv_4=\beta_1 \lambda_1 v_1+\beta_2 \lambda_2 v_2+\beta_3 \lambda_3 v_3
\end{equation*}
which implies $c=\beta_1=\beta_2=\beta_3=0$, but then $\varphi_0=Bf=0$ which is a contradiction since eigenvectors should be different from $0$. 
\end{proof}

\begin{corollary}
The property \eqref{spectral of F} holds for our model.
\end{corollary}

\hspace{-1.9em} Next we prove the analytic semigroup property for operator $L_0$.

\begin{prop} \label{A_0 is a closed operator}
Given the setting of the nonsymmetric model the operator $L_0=D_yF(0,\lambda_0)$ is closed as a mapping in $Z$, with dense domain of definition $D(L_0)=X$.
\end{prop}

\begin{proof}
Density of the domain of definition follows from the Proposition~\ref{properties of X and Z}. Recall that the eigenvalues of $L_0$ are $\tilde{z}_j(n,\lambda_0)$ (cf. Proposition~\ref{degeneracy caused by the symmetry}) so we may choose $\mu \in \RR$ with $\mu \neq z_j(n,\lambda_0)$ for $\forall n \in \ZZ$ and $j=1,...,4$, , i.e. $\mu$ is not an eigenvalue of $L_0$.

Let $\{\varphi_n\} \subset D(L_0)$ with $\varphi_n \rightarrow \varphi$ and $L_0\varphi_n \rightarrow \psi$ in $Z$ for some $\varphi, \psi \in Z$.
Let $\hat{\varphi}_{n,m}(k) \in \CC^4$ denote the Fourier coefficients of $\varphi_n-\varphi_m$ then
\begin{equation*}
\begin{split}
||\varphi_n-\varphi_m||_{H^2}^2 =& \sum_{k \in \ZZ}(1+k^2+k^4) \left|\left( \widetilde{M}(k,\lambda_0)-\mu \right)^{-1}\left( \widetilde{M}(k,\lambda_0)-\mu \right)\hat{\varphi}_{n,m}(k)\right|^2 \lesssim \\
\lesssim & \sum_{k \in \ZZ} \left| \left( \widetilde{M}(k,\lambda_0)-\mu \right)\hat{\varphi}_{n,m}(k)\right|^2 = ||(L_0-\mu)(\varphi_n-\varphi_m)||_{L^2}^2 \lesssim \\
\lesssim & ||L_0(\varphi_n-\varphi_m)||_{L^2}^2 + |\mu| \cdot ||(\varphi_n-\varphi_m)||_{L^2}^2
\end{split}
\end{equation*}
where we have used the estimate $|| ( \widetilde{M}(k,\lambda_0)-\mu )^{-1} ||_\infty \lesssim \dfrac{1}{k^2}$ for $k$ large enough. The proof of which is analogous to that of Proposition~\ref{properties of coefficient matrices} $(ii)$. Now since $\{\varphi_n\}$ and $\{L_0\varphi_n\}$ are Cauchy sequences in $Z$, the above inequality shows that $\{\varphi_n\}$ is a Cauchy sequence in $[H^2(0,1)]^4$ as well and hence converges there. The latter implies that $\varphi \in [H^2(0,1)]^4$ and $\varphi_n \rightarrow \varphi$ in $H^2$. Using the continuous embedding $H^2(0,1)\hookrightarrow C^1([0,1])$ and passing to limits in equalities $\varphi_n(0)=\varphi_n(1)$ and $\varphi_n'(0)=\varphi_n'(1)$ we obtain $\varphi(0)=\varphi(1)$ and $\varphi'(0)=\varphi'(1)$, implying that $\varphi \in [H^2_{per}(0,1)]^4$. Moreover $\varphi \in D(L_0)=X$ since $\{\varphi_n\} \subset X$ (cf. Proposition~\ref{properties of X and Z} $(ii)$). 

Remains to use $L_0 \in L(X,Z)$ to get $L_0\varphi_n \rightarrow L_0\varphi$ in $Z$. The latter implies $\psi = L_0 \varphi$, thus proving that $L_0$ is closed.  
\end{proof}

\begin{prop} \label{analytic semigroup}
Given the setting of the nonsymmetric model, for the operator $L_0=D_yF(0,\lambda_0)$ we have
\begin{enumerate}
\item[\textit{(i)}] $\exists \ \omega>0, \gamma>0 \quad \text{s.t.} \quad S=\{\xi \in \CC \ / \ |arg(\xi - \omega)|< \frac{\pi}{2}+\gamma\} \subset \rho(L_0)$
\item[\textit{(ii)}] $\exists \ C>0 \quad \text{s.t.} \quad \left|\left|(\xi-L_0)^{-1}\right|\right| \leq \dfrac{C}{|\xi|} \qquad \forall \xi \in S$
\end{enumerate}
\end{prop}
\begin{remark}
In the theorem $\rho(L_0)$ denotes the resolvent set of $L_0$ and since we consider $L_0:D(L_0)\subset Z \rightarrow Z$ the norm in $(ii)$ denotes the $Z \rightarrow Z$ operator norm. 
\end{remark}

\begin{figure}[h]
\includegraphics[width=6cm]{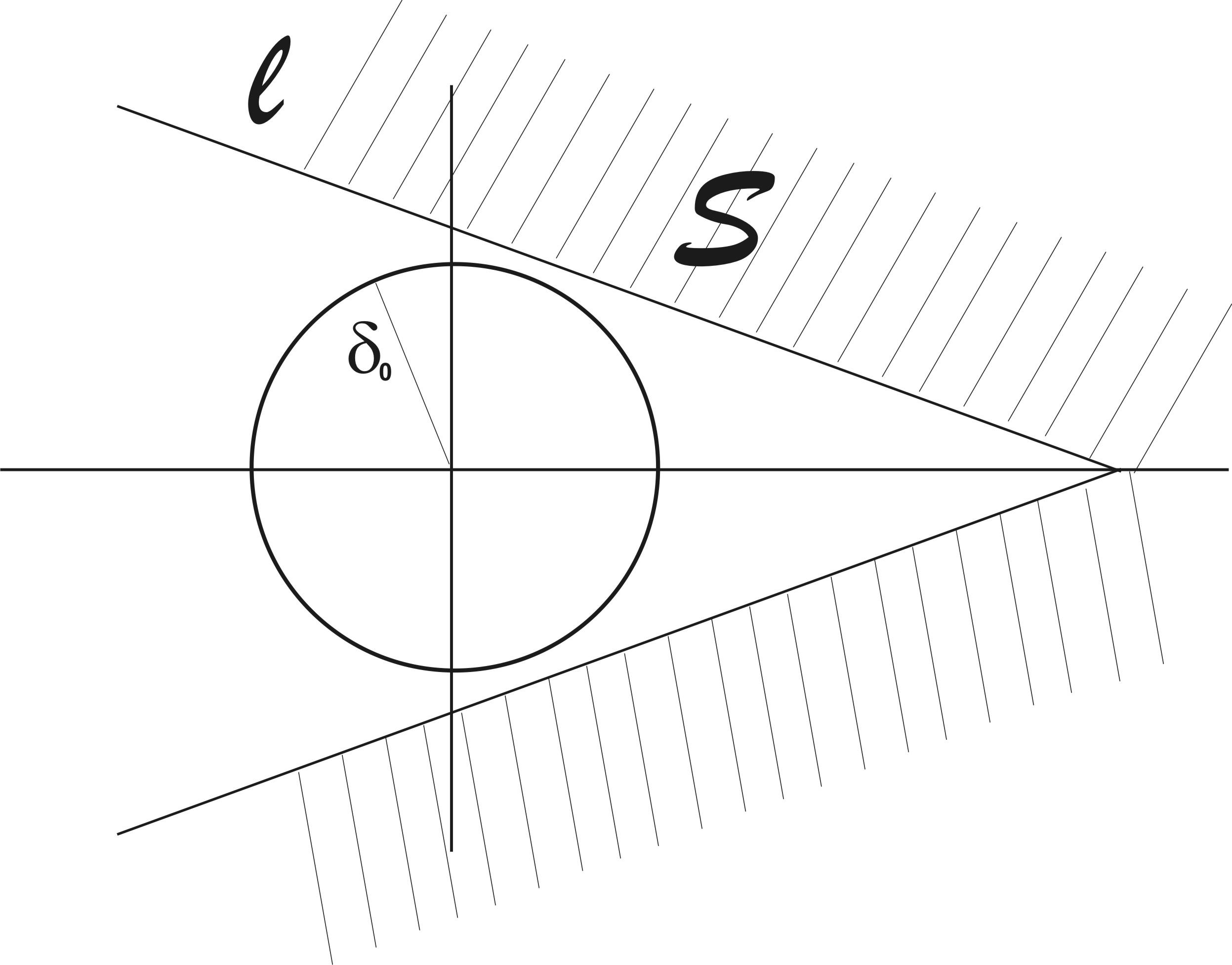}
\centering
\end{figure}

\begin{proof}
\begin{enumerate}
\item[\textit{(i)}]
For the ease of the notation we set $a:=\frac{4\pi^2 \varepsilon}{\lambda_0^2}$, $b:=\frac{2\pi i}{\lambda_0}$ and $B_n := \widetilde{M}(n,\lambda_0)=(-an^2-bnU-DA)$ (cf. \eqref{eigenvalue problem 2v}), then
\begin{equation*}
(\xi-L_0)\varphi=\sum_{n \in \ZZ}(\xi-B_n)\hat{\varphi}(n)e^{i2\pi n x}
\end{equation*}

Proving the invertibility of $\xi-L_0$ is equivalent to proving the invertibility of $\xi-B_n$ for all $n \in \ZZ$. Therefore we should construct the set $S$ so that all the eigenvalues of matrices $\{B_n\}_{n \in \ZZ}$ lie outside of it. Let $l$ be one of the boundary lines of the set $S$ (as shown in the picture) and be given by the equation $y=\alpha x +\beta$ where $\alpha , \beta \in \RR$ would be chosen below.
\begin{equation*}
(\xi-B_n)^{-1}=(an^2+\xi)^{-1}\cdot (Id-T_{n,\xi})^{-1} \qquad \text{where} \quad T_{n,\xi}=\frac{-bn-DA}{an^2+\xi}
\end{equation*}

So first of all we would like to choose $\alpha, \beta$ so that
\begin{equation} \label{1st constraint on S}
\{ -a n^2\}_{n \in \ZZ} \nsubseteq S
\end{equation}
The next constraint should make sure that the inverse $(Id-T_{n,\xi})^{-1}$ exists for all $n \in \ZZ$ and $\xi \in S$. We start by showing the existence for $|n|$ large enough. Choose $\alpha, \beta$ so that the distance between $S$ and the point $-an^2$ is of order $n^2$, namely we show
\begin{equation} \label{distance condition}
\exists \ c>0 \quad \text{s.t.} \quad \text{dist}(-an^2,S) \geq cn^2 \qquad \forall n \in \ZZ
\end{equation}
Note that 
\begin{equation*}
\text{dist}(-an^2,S)=\text{dist}(-an^2,l)=\frac{|-\alpha a n^2+\beta|}{\sqrt{\alpha^2+1}}\geq \frac{-\alpha a}{\sqrt{\alpha^2+1}} \ n^2
\end{equation*}
provided $\alpha<0$ and $\beta>0$. In particular let $\alpha=-1$ and $\beta>0$ then \eqref{distance condition} holds with $c=a/\sqrt{2}$. As a consequence we get
\begin{equation*}
\begin{split}
||T_{n,\xi}||=&\frac{1}{|an^2+\xi|}||bnU+DA|| \leq \frac{1}{\text{dist}(-an^2,S)}||bnU+DA|| \\
\lesssim & \frac{1}{n^2}(|b|\cdot |n| \ ||U||+||DA||) \qquad \forall n \in \ZZ, \xi \in S
\end{split}
\end{equation*}
This estimate shows that $\exists \ N_0$ s.t. $||T_{n,\xi}|| < 1/2 \quad \forall |n|>N_0$, which implies the existence of $(Id-T_{n,\xi})^{-1}$ for any $|n|>N_0, \xi \in S$. To treat the case $|n|\leq N_0$ we change $S$ so that its points become far away from the origin. Let $\delta_0:=\frac{4\pi N_0}{\lambda_0}||U||+2||DA||$ and choose $\beta>\delta_0 \sqrt{2}>0$, then $\text{dist}(0,S)=\text{dist}(0,l)=\beta / \sqrt{2}>\delta_0$. Now for any $|n| \leq N_0$ and $\lambda \in S$, using $|an^2+\xi|\geq |\xi|$ because $a>0$ we obtain
\begin{equation*}
||T_{n,\xi}|| \leq \frac{1}{|\xi|}(|b| \ N_0 \ ||U||+||DA||)<\frac{1}{2} \qquad \text{since} \quad |\xi|>\delta_0
\end{equation*}
which again implies existence of the inverse. Thus choosing $\alpha=-1$ and $\beta > \delta_0 \sqrt{2}$ we conclude the proof since \eqref{1st constraint on S} is obviously satisfied for this choice.  
\item[\textit{(ii)}] From the Neumann series theorem we get
\begin{equation*}
\left|\left|(Id-T_{n,\xi})^{-1}\right|\right|\leq \sum_{k=0}^\infty ||T_{n,\xi}||^k < \sum_{k=0}^\infty \frac{1}{2^k}=2
\end{equation*}
hence we obtain $\left|\left|(\xi-B_n)^{-1}\right|\right| \leq \dfrac{2}{|\xi|} \qquad \forall n \in \ZZ$, which in turn implies
\begin{equation*}
\left|\left|(\xi-L_0)^{-1}\varphi \right|\right|^2 = \sum_{n \in \ZZ} \left|(\xi-B_n)^{-1} \hat{\varphi}(n)\right|^2 \leq \frac{4}{|\xi|^2} \sum_{n \in \ZZ} |\hat{\varphi}(n)|^2 = \frac{4}{|\xi|^2} ||\varphi||^2
\end{equation*}
showing that $\left|\left|(\xi-L_0)^{-1}\right|\right| \leq \dfrac{2}{|\xi|}$ for any $\xi \in S$.
\end{enumerate}
\end{proof}

\begin{corollary}
The closed operator $L_0:Z\rightarrow Z$ with $D(L_0)=X$ generates an analytic semigroup $\{e^{tL_0}\}_{t\geq 0}$ on $Z$.
\end{corollary}
\begin{proof}
The proof readily follows from the above proposition and Theorem $12.31$ of \cite{renardy}.
\end{proof}

\hspace{-1.8em}Remains to show that the semigroup is compact on $Z$ for $t>0$.

\begin{prop} \label{regularity of semigroup}
Given the setting of the nonsymmetric model and $L_0=D_yF(0,\lambda_0)$ for $\forall \ t>0 \ \exists C>0$ s.t.
\begin{equation*}
\left|\left|e^{tL_0}\varphi \right|\right|_{H^2} \leq C ||\varphi||_{L^2} \qquad \forall \ \varphi \in Z
\end{equation*}
\end{prop}

\begin{proof}
First we take $\varphi \in D(L_0)=X$ then by the properties of $C_0$-semigroup $\psi(t):=e^{L_0t}\varphi \in D(L_0)$ and $t \mapsto \psi(t) \in Z$ is differentiable. For fixed $t>0$ we expand (with $'=\frac{d}{dt}$)
\begin{equation*}
\psi(t,x)=\sum_{n \in \ZZ} \hat{\psi}_n(t) e^{i2\pi n x} \qquad \psi'(t,x)=\sum_{n \in \ZZ} \hat{r}_n(t) e^{i2\pi n x}
\end{equation*}
Then one can note by Cauchy-Schwartz that
\begin{equation*}
\begin{split}
\left| \dfrac{\hat{\psi}_n(t+h)-\hat{\psi}_n(t)}{h}-\hat{r}_n(t) \right| = & \left| \int_0^1 \left( \dfrac{\psi(t+h,x)-\psi(t,x)}{h}- \psi'(t,x) \right) e^{-i2\pi n x} \diff x \right| \lesssim \\
\lesssim & \left|\left| \frac{\psi(t+h)-\psi(t)}{h} -\psi '(t)  \right|\right|_{L^2}\longrightarrow 0 \qquad \text{as} ~ h \rightarrow 0
\end{split}
\end{equation*}
which shows that $t \mapsto \hat{\psi}_n(t) \in \CC^4$ is differentiable with $\hat{\psi}'_n(t)=\hat{r}_n(t)$.

It is known that $\psi '(t)=L_0 \psi(t)$ and $\psi(0)=\varphi$. Using the definition of $L_0$ and uniqueness of Fourier expansion, for each $n \in \ZZ$ we obtain the following system of first-order linear ODE's with constant coefficients:
\begin{equation*}
\begin{cases}
\hat{\psi}'_n(t)=B_n \hat{\psi}_n(t) \\
\hat{\psi}_n(0)=\hat{\varphi}_n
\end{cases}
\end{equation*}
where $\varphi=\sum_{n \in \ZZ}\hat{\varphi}_n e^{i2\pi n x}$ and $B_n$ is given in the proof of Proposition~\ref{analytic semigroup}$(i)$. Now the solution is given by $\hat{\psi}_n(t)=e^{t B_n}\hat{\varphi}_n$. Thus we obtain the representation
\begin{equation} \label{representation for semigroup}
e^{tL_0}\varphi (x) = \sum_{n \in \ZZ} e^{t B_n}\hat{\varphi}_n e^{i2\pi n x}
\end{equation}
Recall that if the matrices $T$ and $R$ commute then $e^{T+R}=e^T e^R$. Hence $e^{t B_n}=e^{-atn^2}e^{-t(bnU+DA)}$,
\begin{equation*}
\begin{split}
\left|\left| e^{tL_0}\varphi \right|\right|_{H^2}^2 \lesssim & \sum_{n \in \ZZ} n^4 e^{-2atn^2} \left| e^{-t(bnU+DA)} \hat{\varphi}_n \right|^2 \leq \sum_{n \in \ZZ} n^4 e^{-2atn^2} e^{2t(c_1|n|+c_2)} | \hat{\varphi}_n |^2 \lesssim \\
\lesssim & \sum_{n \in \ZZ} | \hat{\varphi}_n |^2 = ||\varphi ||_{L^2}^2
\end{split}
\end{equation*}
where $c_1=|b|\cdot ||U||, c_2=||DA||$ and we used the fact $a>0$ which implies boundedness of the sequence $\{n^4 e^{-2atn^2} e^{2t(c_1|n|+c_2)}\}_{n \in \ZZ}$.

For the general case $\varphi \in Z$ we use density of $X$ in $Z$ (cf. Proposition~\ref{properties of X and Z}) to get a sequence $\{\varphi_n\} \subset X$ with $\varphi_n \rightarrow \varphi$ in $Z$. But then for a fixed $t>0$ continuity of the semigroup on $Z$ implies $e^{tL_0}\varphi_n \rightarrow e^{tL_0}\varphi$ in $Z$. Using the inequality obtained above we see $\left|\left| e^{tL_0}\varphi_n -e^{tL_0}\varphi_m \right|\right|_{H^2} \lesssim ||\varphi_n -\varphi_m||_{L^2}$ and so $\{e^{tL_0}\varphi_n\}$ is a Cauchy sequence in $[H^2(0,1)]^4$ hence converges there: $\exists \ \psi \in [H^2]^4$ s.t. $e^{tL_0}\varphi_n \rightarrow \psi$ in $H^2$. The latter implies $\psi = e^{tL_0}\varphi$ and it remains to pass to limits in the inequality $\left|\left|e^{tL_0}\varphi_n \right|\right|_{H^2} \lesssim ||\varphi_n||_{L^2}$ as $n\rightarrow \infty$.  
\end{proof}

\begin{corollary}
The semigroup $e^{tL_0}$ is compact on $Z$ for any $t>0$.
\end{corollary}

\begin{proof}
Take any bounded sequence $\{\varphi_n\} \subset Z$ then by Proposition~\ref{properties of X and Z} and ~\ref{regularity of semigroup} $\{e^{tL_0}\varphi_n\}$ is a bounded sequence in $[H^2(0,1)]^4\hookrightarrow_C [L^2(0,1)]^4$ hence admits a convergent subsequence which we don't relabel: $\exists \ \varphi$ s.t. $e^{tL_0}\varphi_n\rightarrow \varphi$ in $L^2$. Remains to note that $\varphi \in Z$ since $Z$ is closed.
\end{proof}
Thus the condition \eqref{semigroup} is satisfied for our model. Finally the condition \eqref{nonvanishing speed} can be checked numerically (e.g. by means of Proposition~\ref{nonvanishing speed prop} and \eqref{nonvanishing speed formula}). As a result we obtain $Re z_1'(\lambda_0) \approx 0.896648 $. Therefore the Theorem~\ref{Hopf Bifurcation Theroem} applies to the non-symmetric model.

Finally to determine the type of the bifurcation we note that \eqref{long formula} reads for our case
\begin{equation*}
\begin{split}
D^2_{rr}\Phi^0=-&\left\langle DQ(\bar{\varphi}_0)\cdot \left(    (2i\kappa_0-L_0)^{-1}DQ(\varphi_0)\varphi_0 \right), \varphi_0^* \right\rangle \\
+&2 \left\langle DQ(\varphi_0) \cdot \left( L_0^{-1}DQ(\varphi_0)\bar{\varphi}_0 \right), \varphi_0^* \right\rangle
\end{split}
\end{equation*}
where
\begin{equation} \label{DQ}
DQ(y)=
\begin{pmatrix}
-2c_1y_1 & 0 & 0 & 2c_4y_4 \\
2c_1y_1 & -2c_2y_2 & 0 & 0 \\
0 & 2c_2y_2 & -2c_3y_3 & 0 \\
0 & 0 & 2c_3y_3 & -2c_4y_4
\end{pmatrix}
\end{equation}
Note that $\varphi_0=e^{i2\pi x}v_0$ (cf. \eqref{eigenvalue i kappa_0 with eigenvector phi_0}). Next we pass from the above bilinear dual pairing given by $\langle u, v\rangle=\int_0^1 u \cdot v \diff x$ to the $L^2$ scalar product. To that end we use \eqref{A_0^*} then $\varphi_0^*=e^{i2\pi x}w_0$ where $\widetilde{M}(1,\lambda_0)^* w_0=i\kappa_0 w_0$. Finally we obtain
\begin{equation}
\begin{split}
D^2_{rr}\Phi^0=-&\left\langle DQ(\bar{v}_0) \left( 2i\kappa_0-\widetilde{M}(2,\lambda_0) \right)^{-1}DQ(v_0)v_0, \ w_0  \right\rangle_2 \\
+&2 \left\langle DQ(v_0) \widetilde{M}(0,\lambda_0)^{-1} DQ(v_0)\bar{v}_0, \ w_0  \right\rangle_2
\end{split}
\end{equation}
where now $\langle\rangle_2$ denotes the Euclidean scalar product on $\CC^4$. Note that $\widetilde{M}(0,\lambda_0)=-DA$ and by \eqref{DQ} the sum of the components of the vector $DQ(v_0)\bar{v}_0$ is $0$. Hence $(-DA)^{-1}DQ(v_0)\bar{v}_0$ is well defined (cf. Proposition~\ref{nonresonance condition Proposition}).

\begin{figure}[h] \hspace{10mm}
    \includegraphics[width=6.5cm, height=4cm]{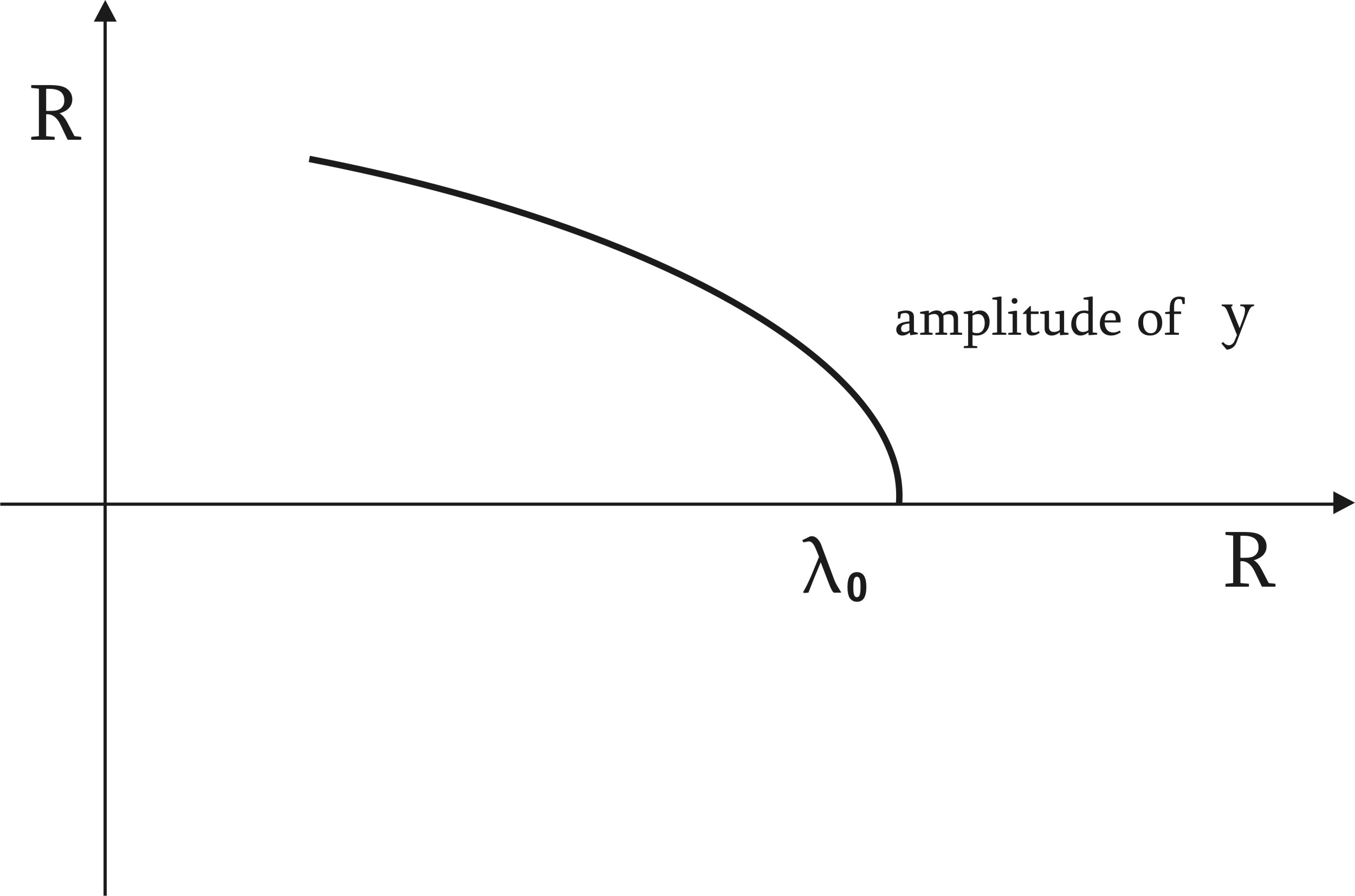} \hspace{20mm}
\includegraphics[width=6.5cm, height=4cm]{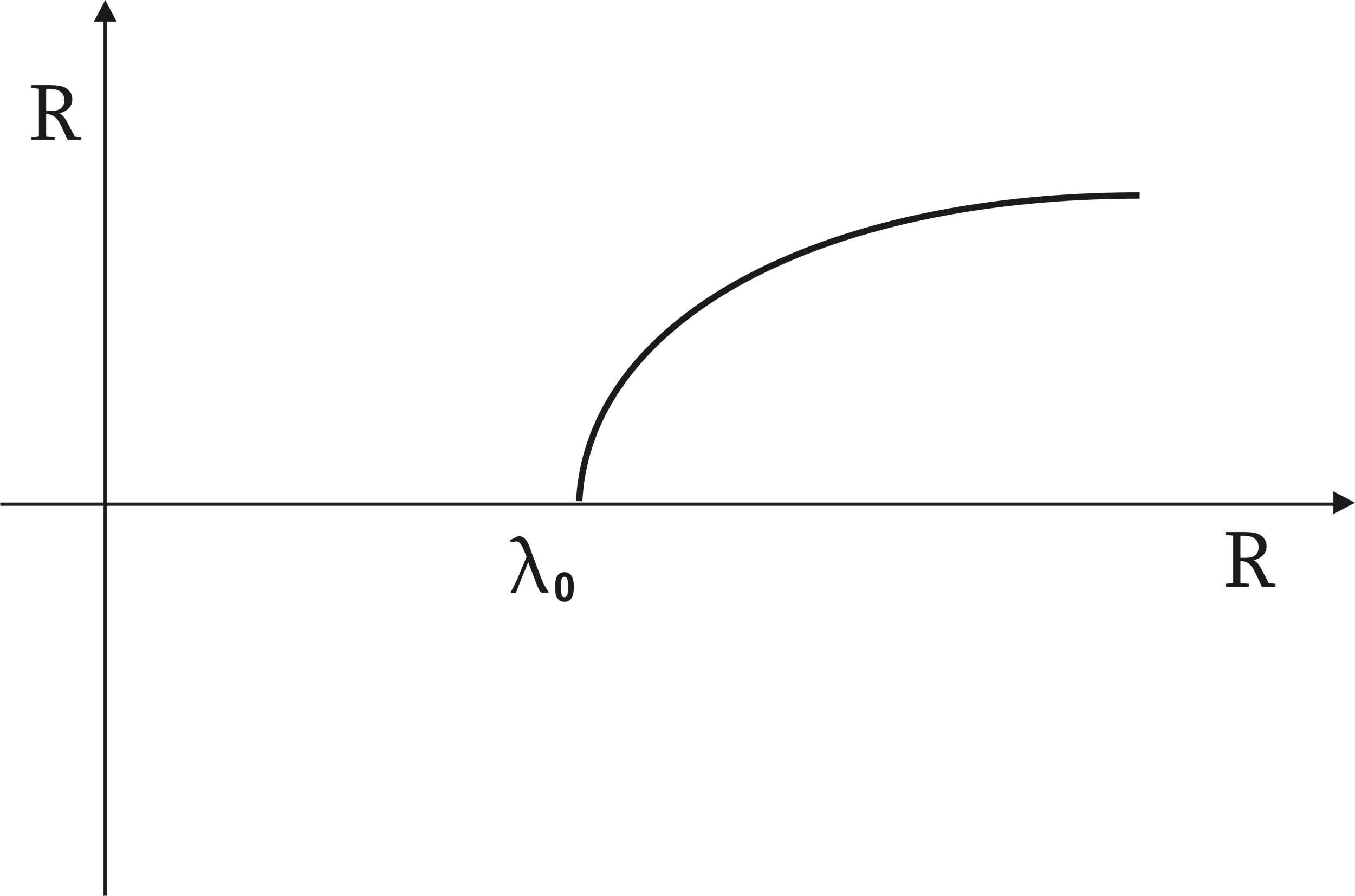}
\end{figure}

If $Re D^2_{rr}\Phi^0 \neq 0$, we have a sub- or supercritical "pitchfork" bifurcation of periodic solutions sketched in figure above. We sketch only one branch which represents the amplitude $\max_{t\in \RR}||y(t)||$ (with norm in $X$) of the bifurcating periodic solution.

\section{Proof of the main results: Symmetric model} \label{symmetric model}
In this section we are in the setting of the Theorem \ref{main theorem, symmetric}, so in particular $M=M_{s}$. And we aim to show that the Hopf Bifurcation Theorem in the case of multiple eigenvalues (cf. Theorem \ref{Hopf Bifurcation at multiple eigenvalues}) applies to our model. First of all we rewrite \eqref{main problem evolution symmetric} in the form of \eqref{evolution equation ME}

\begin{equation} \label{operators in the symmetric case}
\begin{split}
L &= \frac{1}{\lambda_0}U\partial_x+DA-\frac{\varepsilon_0}{\lambda_0^2}\partial_x^2 \\
B(\lambda) &=\left( \frac{1}{\lambda}-\frac{1}{\lambda_0} \right)U\partial_x \\
G(\lambda,y) &=Q(y)
\end{split}
\end{equation}
In our case $\lambda_0$ plays the role of $0$, since we consider bifurcation from $(0,\lambda_0)$. As before we consider the evolution in the space $Z$ given by \eqref{our X, Z}. Now using the Taylor expansion $\frac{1}{\lambda}-\frac{1}{\lambda_0} = -\frac{1}{\lambda_0^2}(\lambda-\lambda_0)+\frac{1}{\lambda_0^3}(\lambda-\lambda_0)^2+...$ we see
\begin{equation} \label{operators B and G}
\begin{split}
B=&-\frac{1}{\lambda_0^2}U\partial_x, \qquad B_2=\frac{1}{\lambda_0^3}U\partial_x, \quad \text{...} \quad , B_n=\frac{(-1)^n}{\lambda_0^{n+1}}U\partial_x \\
&G(y)=Q(y), \quad G_2=G_3=...=0 \\
C_0&=-\frac{\varepsilon_0}{\lambda_0^2}\partial_x^2 + Id, \quad B_0=\frac{1}{\lambda_0}U\partial_x+DA - Id
\end{split}
\end{equation}
\begin{remark}
We took $\varepsilon=\varepsilon(\lambda)= \frac{\lambda^2}{\lambda_0^2}\varepsilon_0$ (which is not really a restriction since in the region where $\lambda$ is close to $\lambda_0$ we get that $\varepsilon(\lambda)\sim \varepsilon_0$, i.e. it is a small number) in Theorem~\ref{main theorem, symmetric} in order to make sure that $B(\lambda)$ (and hence $B,B_2,...$) doesn't contain the second order derivative term $\partial_x^2$. Which would imply that the estimates \eqref{assumption on B_0} and \eqref{final assumption} hold true. (For the proof of this kind of results, e.g. with $\alpha=1/2$ we refer to \cite{henry}).
\end{remark}

The condition \eqref{series assumption} is satisfied, because once we have a bound with $c_2$ in \eqref{operators B and G} we can take $c_{3,j}=c_{5,j}=\frac{1}{|\lambda_0|^j}$.

Note that $L=-D_yF(0,\lambda_0)$ so the previously proven propositions apply to this operator. In particular $-i\kappa_0$ is an eigenvalue of $L$ with a two dimensional eigenspace $\{e^{i2 \pi x}v_0, e^{-i2 \pi x}Pv_0\}$ (cf. Proposition~\ref{degeneracy caused by the symmetry}). The eigenspace is exactly two-dimensional since the eigenvalues $z_j$ cross the imaginary axis at $-i\kappa_0$ two times (cf. Figure~\ref{fig:symmetric model}). Thus we may take $\mu_0=-\kappa_0>0$ then invoking Proposition~\ref{nonresonance condition Proposition} we see that \eqref{assumption on A ME} is satisfied. So we see that $r=2$ and using Proposition~\ref{adjoint of A_0} we get
\begin{equation*}
\ker (L-i\mu_0)= span \{\underbrace{e^{i2\pi x}v_0}_{\varphi_1}, \underbrace{e^{-i2\pi x}Pv_0}_{\varphi_2}\}, \quad \ker (L^*+i\mu_0)= span \{\underbrace{e^{i2\pi x}w_0}_{\varphi_1^*}, \underbrace{e^{-i2\pi x}Pw_0}_{\varphi_2^*}\}
\end{equation*} 
where $(\varepsilon_0 k_0^2+ik_0U+DA)v_0=i\mu_0 v_0$ and $(\varepsilon_0 k_0^2-ik_0U+A^*D^*)w_0=-i\mu_0 w_0$ (see also \eqref{lambda0}). To obtain the normalization in \eqref{dual bases} we require $\langle v_0, w_0 \rangle_2=\frac{1}{2\pi}$.
\begin{prop}
$i\mu_0$ \ is a semisimple eigenvalue of $L$ of multiplicity $2$. 
\end{prop}
\begin{proof}
We have already established the multiplicity. Now the semisimplicity in our context means that $\ker (i\mu_0-L) \cap \Re{(i\mu_0-L)} = \{0\}$. Take any $\varphi$ from the intersection, then $\exists \alpha, \beta \in \CC$ and $f \in D(L)$ s.t. $\varphi=\alpha \varphi_1 + \beta \varphi_2 = (L-i\mu_0)f$, i.e.
\begin{equation} \label{semisimplicity expansion}
\alpha e^{i2\pi x}v_0 + \beta e^{-i2\pi x}Pv_0=\sum_{n\in \ZZ} (\varepsilon_0 n^2k_0^2+ink_0U+DA-i\mu_0)\hat{f}(n)
\end{equation}
Equating the Fourier coefficients we see $\hat{f}(n)=0 \quad \forall n \in \ZZ \backslash \{1,-1\}$ (since the corresponding matrices are invertible) and
\begin{equation*}
\alpha v_0 = M_1 \hat{f}(1), \qquad \beta Pv_0 = M_{-1}\hat{f}(-1) 
\end{equation*}
where $M_1$ and $M_{-1}$ are the matrices on RHS of \eqref{semisimplicity expansion} for correspondingly $n=1$ and $n=-1$. Now $v_0$ (resp. $Pv_0$) is in the kernel of $M_1$ (resp. $M_{-1}$). But numerical computations (e.g. take $\varepsilon_0=\delta=0.001$) show that the matrices $M_1, M_{-1}$ have $4$ distinct eigenvalues hence there exist  bases of $\CC^4$ consisting of corresponding eigenvectors. E.g. let $\{w_1, w_2, w_3, v_0\} \subset \CC^4$ be the one for $M_1$ with eigenvalues $\{\sigma_1, \sigma_2, \sigma_3, 0\}$, then expanding $\hat{f}(1)$ in this basis we get $\alpha v_0 = c_1 \sigma_1 w_1+c_2 \sigma_2 w_2+c_3 \sigma_3 w_3 $ and by independence we deduce that $\alpha=0$. Similarly one obtains $\beta=0$ and thus $\varphi=0$.
\end{proof}
Finally we compute matrices in \eqref{corresponding matrices}:
\begin{equation*}
\begin{split}
(B\varphi_1, \varphi_1^*)&=-\frac{ik_0^2}{2\pi} \langle Uv_0, w_0\rangle \\
(B\varphi_2, \varphi_2^*)&=\frac{ik_0^2}{2\pi} \langle UPv_0, Pw_0\rangle =(B\varphi_1, \varphi_1^*)
\end{split}
\end{equation*}
The last equality holds true since $UP=-PU$. So, after dropping the two redundant equations (which are conjugates of the first two) we obtain
\begin{equation*}
P_0B \leftrightarrow
\begin{pmatrix}
a & 0\\
0 & a \\
\end{pmatrix} \qquad \text{with} \quad
a=-ik_0^2\langle Uv_0, w_0\rangle
\end{equation*}
In computing coefficients $a_{ijk}^l$ of \eqref{formula for E^3} we use the following formulas coming from Fourier expansion
\begin{equation*}
\begin{split}
L^{-1}f=\sum_{n \in \ZZ} \left( \varepsilon_0 n^2k_0^2+ink_0U+DA \right)^{-1} \hat{f}(n)e^{i2\pi nx} \\
(L-2i\mu_0)^{-1}f=\sum_{n \in \ZZ} \left( \varepsilon_0 n^2k_0^2+ink_0U+DA-2i\mu_0 \right)^{-1} \hat{f}(n)e^{i2\pi nx}
\end{split}
\end{equation*}
where the inverses exist in view of the Proposition~\ref{nonresonance condition Proposition}. Finally we see that for our nonlinearity $Q$ the corresponding polar form is given by
\begin{equation*}
G^{(2)}(u,v)=\begin{pmatrix}
c_2u_4v_4-c_1u_1v_1 \\
c_1u_1v_1-c_2u_2v_2 \\
c_2u_2v_2-c_1u_3v_3 \\
c_1u_3v_3-c_2u_4v_4
\end{pmatrix}
\end{equation*}
Considering a particular choice of the nonlinearity with $c_1=1, c_2=0$ and by taking $\varepsilon_0=\delta=0.001$  we compute numerically the vector corresponding to third-order terms (cf. \eqref{formula for E^3})
\begin{equation} \label{numerical value of E^3}
\begin{split}
E^{(3)}(v)\leftrightarrow 2\big{(}(316.127 + 912.071 i) h_1^2 \bar{h}_1 &+ (0.0660957 + 0.175946 i) h_1 h_2 \bar{h}_1 - \\ - (316.128 + 912.074 i) h_1 h_2 \bar{h}_2 &+ (0.00475099 + 0.0576605 i) h_2^2 \bar{h}_2, \\ 
(0.00475099 + 0.0576605 i) h_1^2 \bar{h}_1 &- (316.128 + 912.074 i) h_1 h_2 \bar{h}_1 + \\ + (0.0660957 + 0.175946 i) h_1 h_2 \bar{h}_2 &+ (316.127 + 912.071 i) h_2^2 \bar{h}_2\big{)}
\end{split}
\end{equation}
Further we get $a \approx -0.0000324659 - 0.0406768 i$. Now as was described in Section~\ref{Hopf bifurcation at multiple eigenvalues, section} we pass from $2$ complex equations to $4$ real ones. This changes the matrix data as follows
\begin{equation} \label{numerical data}
\Upsilon=
\begin{pmatrix}
0 & 0 & 0 & 0 \\
-1 & 0 & 0 & 0 \\
0 & 0 & 0 & 0 \\
0 & 0 & -1 & 0
\end{pmatrix}, \quad
P_0B=
\begin{pmatrix}
Re(a) & 0 & 0 & 0 \\
Im(a) & 0 & 0 & 0 \\
0 & 0 & Re(a) & 0 \\
0 & 0 & Im(a) & 0
\end{pmatrix}, \quad
v=
\begin{pmatrix}
x_1 \\
y_1 \\
x_2 \\
0
\end{pmatrix}
\end{equation}
So we took $j=2$. Finally we find that with the above data 
\begin{equation}
x_1=x_2 \approx 0.0756877 , \quad y_1=0 \quad \text{and} \quad \rho \approx -24.64899
\end{equation}
solves the equation \eqref{necessary equation multiple eigenvalue}. Remains to check the condition \eqref{existence condition}. The corresponding determinant turns out to be approximately $6.28814 \times 10^{-7}$. Although this number is quite small, nevertheless we have done the computations with an accuracy $10^{-17}$, also in view of the order of parameters $\varepsilon_0, \delta$ and the very small magnitude of the entries of the matrix in \eqref{existence condition}, we can surely say that it is different from $0$. 

Thus the required condition is satisfied and hence the Theorem~\ref{Hopf Bifurcation at multiple eigenvalues} applies to the model.

\pagebreak


\begin{thebibliography}{13}

\bibitem{deutsch}
U. B\"{o}rner, A. Deutsch, H. Reichenbach and M. B\"{a}r, \textit{Rippling patterns in aggregates of myxobacteria arise from cell-cell collisions}, Physical Review Letters, Vol. 89 (2002), 078101

\bibitem{dworkin}
M. Dworkin and D. Kaiser eds., "Myxobacteria II", American Society for Microbiology (AMS) Press, 1993.

\bibitem{gierer}
A. Gierer and H. Meinhardt, \textit{A theory of biological pattern formation}, Kybernetik, Vol. 12 (1972), 30-39

\bibitem{henry}
D. Henry, \textit{Geometric theory of semilinear parabolic equations}, Springer, Lecture notes in mathematics, Vol. 840 (1981)

\bibitem{hille}
E. Hille and R. S. Phillips, \textit{Functional analysis and semi-groups}, Amer. Math. Soc., Publ XXXI, Providence, Rhode Island (1957)

\bibitem{neu}
O. Igoshin, J. Neu and G. Oster, \textit{Developmental waves in Myxobacteria: A novel pattern formation mechanism}, Phys. Rev. E, Vol. 7 (2004), 1-11

\bibitem{kato}
T. Kato, \textit{Perturbation theory for linear operators}, New York: Springer Verlag (1966)

\bibitem{kiel1} 
H. Kielh\"{o}fer, \textit{Bifurcation theory: an introduction with applications to PDEs}, Springer, Applied Mathematical Sciences, Vol. 156 (2004)

\bibitem{kiel2}
H. Kielh\"{o}fer, \textit{Hopf Bifurcation at Multiple Eigenvalues}, Archive for Rational
Mechanics and Analysis, Vol. 69 (1979), 53-83

\bibitem{lutscher}
F. Lutscher and A. Stevens, \textit{Emerging patterns in a hyperbolic model for locally interacting cell systems}, J. Nonlinear Science, Vol. 12 (2002), 619-640

\bibitem{velaz} 
I. Primi, A. Stevens and J. J. L. Vel\'azquez,
\textit{Pattern forming instabilities driven by non-diffusive interactions}, Networks and Heterogeneous Media, Vol. 8, No. 1 (2013), 397-432. 

\bibitem{renardy}
M. Renardy and R. C. Rogers, \textit{An introduction to partial differential equations}, Springer, Texts in Applied Mathematics, Vol. 13 (2004)
 
\bibitem{turing}
A. M. Turing, \textit{The chemical basis of morphogenesis}, Philosophical Transactions of the Royal Society of London, Series B, Biological Sciences, 237 (1952), 37-72.
\end{thebibliography}
\end{document}